\documentclass{amsart}

\usepackage{amsmath,amsthm,amssymb,amsfonts}
\usepackage[a4paper]{geometry}
\usepackage{a4wide}
\usepackage{parskip}
\usepackage{graphicx}              
\usepackage{mathpazo}
\usepackage{MnSymbol}
\usepackage[T1]{fontenc}
\usepackage[utf8]{inputenc}
\usepackage{hyperref}
\usepackage{enumitem}
\usepackage{color}
\usepackage{xcolor}
\hypersetup{
    colorlinks,
    linkcolor={red!80!black},
    citecolor={blue!80!black},
    urlcolor={blue!80!black}
}
\usepackage{ mathrsfs } 
\usepackage{marginnote}
\usepackage{tikz}          
\usetikzlibrary{matrix,arrows}  
\usetikzlibrary{graphs}    

\usepackage[margin=10pt,font=normalsize,labelfont=bf, labelsep=endash]{caption}

\colorlet{myred}{red!80!black}
\colorlet{myblue}{blue!80!black}
\colorlet{mygreen}{green!80!black}

\makeatletter
\def\thm@space@setup{\thm@preskip=4pt
\thm@postskip=2pt}
\makeatother

\newtheorem{theorem}{Theorem}[section]
\newtheorem{lemma}[theorem]{Lemma}
\newtheorem{corollary}[theorem]{Corollary}
\newtheorem{fact}[theorem]{Fact}

\newtheorem{proposition}[theorem]{Proposition}

\newtheorem{mainthm}{Theorem}

\theoremstyle{definition} 

\newtheorem{definition}[theorem]{Definition}
\newtheorem{tabel}[theorem]{Table}

\newtheorem{remark}[theorem]{Remark}

\DeclareMathOperator{\Sym}{\mathrm{Sym}}

\newcommand{\ZZ}{{\mathbb Z}}
\newcommand{\NN}{{\mathbb N}}
\newcommand{\PP}{{\mathbb P}}

\newcommand{\RR}{{\mathbb R}}
\newcommand{\CC}{{\mathbb C}}

\newcommand{\cO}{{\mathcal O}}

\newcommand{\cL}{{\mathcal L}}

\newcommand{\cE}{{\mathcal E}}

\newcommand{\ra}{\rightarrow}

\newcommand{\Hom}{\textrm{Hom}}

\newcommand{\bd}[1]{\mathbf{#1}}  
\newcommand{\ol}[1]{\overline{#1}}

\begin{document}

\begin{abstract}
The article covers the unique complex K3 surface with maximal Picard
rank and discriminant four. We discuss smooth, rational curves and
identify generators of its automorphism group with certain Cremona
transformations of $\PP^2$. This gives a geometric perspective of
Vinberg's results \cite{VIN}.
\end{abstract}

\author{Łukasz Sienkiewicz}
\address{Instytut Matematyki UW, Banacha 2, 02-097 Warszawa, Poland}
\email{lusiek@mimuw.edu.pl}
\title{Vinberg's $X_4$ revisited}
\thanks{The author was supported by Polish National Science Center project 2013/08/A/ST1/00804.}
\maketitle

\section*{Introduction}
In  \cite{SHIN} Shioda and Inose proved a classification theorem for complex K3 surfaces with maximal Picard rank in terms of their transcendental lattices. In the course of the proof they discussed two K3 surfaces with maximal Picard rank which are the simplest in the sense that their transcendental lattices have the smallest possible discriminants equal to $3$ and $4$. Then Vinberg in his article \cite{VIN} called these surfaces the most algebraic K3 surfaces. He gave a complete description of automorphism groups of these two surfaces as well as several examples of their birational models. In the article \cite{OGZH} the authors classified these surfaces in terms of ramification locus of some special automorphisms of these surfaces and related them to extremal log Enriques surfaces. Recently in \cite{KAPUSTKOWIEIFERAJNA} the authors used the Hilbert scheme of two points on the K3 surface having discriminant equal to four to solve the question concerning the cardinality of complete families of incident planes in $\PP^5$ originally posed by Morin in \cite{MOR}. This is related to O'Grady's research \cite{Ogrady} on Hyperk{\"a}hler manifolds.\\
This article is devoted to study of a unique K3 surface $X_4$ with Picard rank $20$ and discriminant equal to four. We construct $X_4$ as a double covering $\pi:X_4\ra Y_4$ of a smooth, rational surface $Y_4$. Surface $Y_4$ is defined as a blow up $m:Y_4\ra S_5$ of the Del Pezzo surface $S_5$ of degree $5$ in all intersection points of configuration of $(-1)$-curves on $S_5$. In the book \cite{ALEXNIK} Alexeev and Nikulin describe certain log Del Pezzo surfaces as quotients of K3 surfaces by non-symplectic involutions. In particular, $Y_4$ is a log Del Pezzo surface. So the above construction is an important special case of their procedure.\\
In this paper we describe geometrically generators of the group $\mathrm{Aut}(X_4)$. The following theorem relates automorphisms of $X_4$ and $Y_4$. Here $\sigma$ is a unique nontrivial automorphism of the covering $\pi:X_4\ra Y_4$.
\begin{mainthm}[Theorem \ref{exactaut}]
There exists a short exact sequence of groups
\begin{center}
\begin{tikzpicture}
[description/.style={fill=white,inner sep=2pt}]
\matrix (m) [matrix of math nodes, row sep=3em, column sep=2em,text height=1.5ex, text depth=0.25ex] 
{  1 & \left<\sigma\right>&\mathrm{Aut}(X_4) & \mathrm{Aut}(Y_4)&1    \\} ;
\path[->,font=\scriptsize]  
(m-1-1) edge node[auto] {$ $} (m-1-2)
(m-1-2) edge node[auto] {$ $} (m-1-3)
(m-1-3) edge node[auto] {$ \theta$} (m-1-4)
(m-1-4) edge node[auto] {$ $} (m-1-5);
\end{tikzpicture}
\end{center}
and $\left<\sigma \right>$ is a central subgroup of $\mathrm{Aut}(X_4)$.
\end{mainthm}
\noindent
The covering morphism $\pi$ has ten branch curves $F_{ij}$ for $0\leq i<j\leq 4$ and these are all $(-4)$-curves on $Y_4$ (Proposition \ref{classes}). Let $\Sigma_{\bd{b}}$ be the group of bijections of the set $\bd{b}$ of branch curves of $\pi$. The action of $\mathrm{Aut}(Y_4)$ on $(-4)$-curves of $Y_4$ gives rise to a homomorphism of groups $\mathrm{Aut}(Y_4)\ra \Sigma_{\bd{b}}$. We define a subgroup $G\subseteq \Sigma_{\bd{b}}$ as the image of an injective homomorphism $\Sigma_5\ra \Sigma_{\bd{b}}$ sending $\tau \in \Sigma_5$ to the bijection given by $F_{ij}\mapsto F_{\tau(i)\tau(j)}$. Group $G$ can be also identified with the group of automorphisms of the Petersen graph (\ref{fig1}).
\begin{mainthm}[Theorem \ref{main3}]
The image of the homomorphism $\mathrm{Aut}(Y_4)\ra \Sigma_{\bd{b}}$ is $G$. Moreover, the epimorphism $\mathrm{Aut}(Y_4)\ra G$ admits a section given by homomorphism $G\cong \mathrm{Aut}(S_5)\ra \mathrm{Aut}(Y_4)$ that lifts an automorphism of $S_5$ along birational contraction $m:Y_4\ra S_5$.
\end{mainthm}
\noindent
Let $Q$ be the quadro-quadric Cremona transformation of $\PP^2$. Clearly $Q$ is a birational involution of $\PP^2$. It turns out (Proposition \ref{quadratic}) that $Q$ induces a regular involution $f_Q$ of $Y_4$. Its lift $\tilde{f}_Q$ to an automorphism of $X_4$ induces a hyperbolic reflection of the hyperbolic space associated with $\mathrm{H}^{1,1}_{\RR}(X_4)$ (Corollary \ref{reflectionX_4}). Our last result identifies the hyperbolic reflection induced by $\tilde{f}_{Q}$ with a reflection described in \cite{VIN}.
\begin{mainthm}[Corollary \ref{main4}]
Reflection induced by $\tilde{f}_Q^*$ on the hyperbolic space associated with $\mathrm{H}^{1,1}_{\RR}(X_4)$ is conjugate to reflections contained in $\mathscr{S}_1$ according to Vinberg's notation \cite[Section 2.2]{VIN}. 
\end{mainthm}
\noindent
Let $S_{X_4}$ be the lattice of algebraic cycles of $X_4$ and let $O_+(S_{X_4})$ be the group of orthogonal automorphisms of $S_{X_4}$ that preserve the ample cone. According to Vinberg  \cite[Theorem 2.4]{VIN} we have an isomorphism of groups
$$O_+(S_{X_4})\cong \underbrace{\left(\ZZ/2\ZZ\star...\star\ZZ/2\ZZ\right)}_{5\,times}\rtimes\Sigma_5$$
where five copies of $\ZZ/2\ZZ$ in the free product are generated by reflections in $\mathscr{S}_1$ and $\Sigma_5$ transitively permutes the factors of the free product. Thus Theorems \textbf{B} and \textbf{C} identify factors of this semidirect product in terms of autmorphisms of Del Pezzo surface $S_5$ and quadratic transformations of the plane $\PP^2$.
 
\section{Preliminaries in the theory of K3 surfaces and hyperbolic geometry}\label{1}

\textit{In this paper a K3 surface is a smooth, projective surface $X$ over $\CC$ such that $\Omega^2_X\cong \cO_X$ and $\mathrm{H}^1(X,\cO_X)=0$.} Exhaustive presentation of the theory of K3 surfaces is \cite{Huy}. Fix a K3 surface $X$. It follows that $\mathrm{H}^2(X,\ZZ)$ is a free $\ZZ$-module of rank $22$ and cup product yields to intersection pairing on $\mathrm{H}^2(X,\ZZ)$. The transcendental lattice $T_X$ is the sublattice of $\mathrm{H}^2(X,\ZZ)$ orthogonal to the Neron-Severi lattice $S_X=\mathrm{NS}(X)$ with respect to the cup product. The rank of $S_X$ is called the Picard number of $X$ and is denoted by $\rho(X)$. Let $O_+(S_X)$ be the group of isometries of $S_X$ that preserve the ample cone. Suppose that $\omega_X$ is a nontrivial holomorphic two-form on $X$. We define $U_X=\{\alpha\in O(T_X)\mid (1_{\CC}\otimes_{\ZZ}\alpha)(\omega_X)\in \CC\omega_X\}$. For every lattice $L$ we denote by $L^{\vee}$ its dual. The groups
$$S^{\vee}_X/S_X,\,\mathrm{H}^2(X,\ZZ)/(S_X+T_X),\,T_X^{\vee}/T_X$$
are canonically identified. We call $D_X=\mathrm{H}^2(X,\ZZ)/(S_X+T_X)$ the discriminant group of $X$ and we call its order $|D_X|$ the discriminant of $X$. Note that every automorphism of $X$ yields in an obvious way an element of $O_+(S_X)$ as well as an element of $U_X$. Moreover, every orthogonal transformation of $S_X$ or $T_X$ yields an automorphism of $S^{\vee}_X/S_X$ or $T_X^{\vee}/T_X$ respectively i.e. an automorphism of $D_X$. 
\begin{proposition}[{\cite[Section 1.5, Formula (7)]{VIN}}]\label{cartesian}
We have a cartesian square of abstract groups 
\begin{center}
\begin{tikzpicture}
[description/.style={fill=white,inner sep=2pt}]
\matrix (m) [matrix of math nodes, row sep=3em, column sep=2em,text height=1.5ex, text depth=0.25ex] 
{  \mathrm{Aut}(X)  & O_+(S_{X})      \\
U_X & \mathrm{Aut}(D_{X})\\} ;
\path[->,font=\scriptsize]  
(m-1-1) edge node[auto] {$ $} (m-1-2)
(m-2-1) edge node[below] {$ $} (m-2-2)
(m-1-1) edge node[left] {$ $} (m-2-1)
(m-1-2) edge node[right] {$ $} (m-2-2);
\end{tikzpicture}
\end{center}
\end{proposition}
\noindent
Our results significantly use the theory of elliptic fibrations on K3 surfaces. Recall that a proper and flat morphism with connected fibers $p:X\ra C$ defined on a surface $X$ is an elliptic fibration if and only if its general fiber is an elliptic curve. We extensively use Kodaira classification of singular fibers of elliptic fibrations cf. \cite[Chapter V, Section 7]{CCS}. Next results describe elliptic fibrations on K3 surfaces.

\begin{theorem}[{\cite[Section 3, Theorem 1]{SPS}}]
Let $X$ be a projective K3 surface and $\cL$ a line bundle which is nef and $\cL^2=0$. Then $\cL$ is base point free and the corresponding morphism $\phi_{|\cL|}:X\ra \PP\left(\mathrm{H}^0(X,\cL)\right)$ factors as an elliptic fibration $p:X\ra \PP^1$ followed by a finite morphism $\PP^1\ra \PP\left(\mathrm{H}^0(X,\cL)\right)$.
\end{theorem}

\begin{corollary}[{\cite[Lemma 1.1]{SHIN}}]\label{ellipticK3}
Let $D$ be an effective divisor on a K3 surface $X$. Assume that $D$ is not equal to a multiple of any other divisor. Suppose that $D$ as a scheme has an isomorphism type of a singular fiber of some elliptic fibration. Then there exists an elliptic fibration $p:X\ra \PP^1$ such that $D$ is a singular fiber of $p$.
\end{corollary}

\begin{theorem}[\cite{SH},{ \cite[Lemma 1.3]{SHIN}}]\label{Shioda}
Let $p:X\ra \PP^1$ be an elliptic fibration on a K3 surface and $D_i$ for $1\leq i\leq k$ its singular fibers. Let $m_i$ and $m^{(1)}_i$ denote respectively the number of irreducible components of $D_i$ and the number of irreducible components of $D_i$ having multiplicity one. 
\begin{enumerate}[label=\emph{\textbf{(\arabic*)}}, leftmargin=*]
\item Let $r(p)$ be the torsion-free rank of the group of sections of $p$. Then the following formula holds.
$$\rho(X)=2+r(p)+\sum_{i=1}^k(m_i-1)$$
\item  Moreover, if $r(p)=0$ and $n(p)$ denotes the order of the group of sections of $p$, then the following formula holds.
$$|\mathrm{det}(T_X)|=\frac{\prod_{i=1}^km^{(1)}_i}{n(p)^2}$$
\end{enumerate}
\end{theorem}

\begin{proposition}[{\cite[Theorem 2.3]{KJ}}]\label{elliptics}
Every elliptic fibration on a K3 surface $X$ with $\rho(X)=20$ and with discriminant equal to four or three admits a section.
\end{proposition}
\noindent
In the last part of section \ref{5} we use certain results concerning explicit models of hyperbolic geometry. For the reference cf. \cite[Section 2.2]{REF} or \cite[Section 1.3]{VIN}. Pick $n\in \NN$ and let $E$ be a real vector space of dimension $n+1$ equipped with bilinear pairing $(-,-)$ of signature $(1,n)$. Then the set $\{x\in E\mid (x,x)>0\}$ has two connected components and let $C_+$ be one of them. Then we define $\mathbb{H}^n=C_+/\RR_{>0}$ i.e. we consider vectors in $C_+$ up to positive multiplicative constant. 

\begin{proposition}[{\cite[Section 2.2]{REF}}]
Bilinear form $(-,-)$ induces a Riemannian metric on $\mathbb{H}^n$. This construction gives rise to a Riemannian manifold with constant negative curvature. 
\end{proposition}
\noindent
We call $\mathbb{H}^n$ an $n$-dimensional hyperbolic space.

\begin{corollary}\label{hyperbolic}
Let $X$ be a K3 surface. Pick $E=\mathrm{H}^{1,1}(X)$ and choose $C_+$ to be the connected component of $\{c\in \mathrm{H}^{1,1}(X)\mid (c,c)>0\}$ that contains the ample class of $X$. Then $C_+/\RR_{>0}$ yields a model of a hyperbolic space.
\end{corollary}
\noindent
A reflection in a Riemannian manifold of constant curvature is a nontrivial order two isometry preserving every point inside some totally geodesic hypersurface {\cite[Section 2.2]{REF}}.

\begin{proposition}[{\cite[Section 1.3]{VIN}}]\label{reflection}
For every vector $e\in E$ such that $(e,e)<0$ linear map
$$E\ni x\mapsto x-\frac{2(e,x)}{(e,e)}e\in E$$
induces a reflection of $\mathbb{H}^n$. Moreover, every reflection of $\mathbb{H}^n$ is induced in such a way from a linear map on $E$.
\end{proposition}

\section{Cyclic coverings and $n$-th root of a section}\label{2}

In this section we present material leading to important result concerning lifting of automorphisms. The first result of this section is a part of the folklore and can be extracted from presentation of cyclic coverings in \cite[Section 4.1B]{LAZ}. 

\begin{proposition}\label{universal}
 Let $Y$ be a scheme, $\cE$ be a locally free sheaf on $Y$ and  $s\in \Gamma(Y,\Sym_n(\cE))$ be a global section for some $n\in \NN$. Then there exists a scheme $q:\mathbb{W}_n(\cE,s)\ra Y$ over $Y$ and a section $t_{\cE}\in \Gamma\left(\mathbb{W}_n(\cE,s),q^*\cE\right)$ such that 
\begin{enumerate}[label=\emph{\textbf{(\arabic*)}}, leftmargin=*]
\item $t_{\cE}^n=q^*s$
\item For every morphism $g:X\ra Y$ and a section $t\in \Gamma(X,g^*\cE)$ such that $t^n=g^*s$ there exists a unique morphism $h:X\ra \mathbb{W}_n(\cE,s)$ in the category of schemes over $Y$ such that $t=h^*t_{\cE}$.
\end{enumerate}
\end{proposition}
\noindent
If $\cL$ is a line bundle and $D$ is a divisor of zeros of some section $s \in \Gamma(Y,\cL^{\otimes n})$, then we call $q:\mathbb{W}_n(\cL,s)\ra Y$ a cyclic covering of $Y$ corresponding to $\cL$ branched along $D$. Note that if $Y$ is a complete variety over $\CC$, then the notion of cyclic covering does not depend on the choice of a global section $s\in \Gamma(Y,\cL^{\otimes n})$ having $D$ as the divisor of zeros. We use the notion of cyclic covering in the following special case.

\begin{definition}
Let $Y$ be a smooth and proper variety over $\CC$. We denote by $\omega_Y$ the sheaf of algebraic differential forms of the highest possible degree on $Y$ i.e $\omega_Y=\cO_Y(K_Y)$, where $K_Y$ is the canonical divisor on $Y$. Fix integer $n\in \NN$. Let $D$ be an effective divisor linearly equivalent to $-nK_Y$. Then the cyclic covering of $Y$ branched along $D$ and corresponding to $\omega_Y^{\vee}$ is called the anticanonical cyclic covering of $Y$.
\end{definition}
\noindent
The following proposition is used in the section \ref{5}.
\begin{proposition}\label{lifting}
Let $Y$ be a smooth, proper variety over $\CC$. Let $D$ be an effective divisor such that $D\sim -nK_Y$ for some $n\in \NN$. Denote by $q:X \ra Y$ the anticanonical cyclic covering branched along $D$. Suppose that $f$ is an automorphism of $Y$ such that $f^*D=D$. Then there exists an automorphism $\tilde{f}$ of $X$ such that the following square is commutative.
\begin{center}
\begin{tikzpicture}
[description/.style={fill=white,inner sep=2pt}]
\matrix (m) [matrix of math nodes, row sep=3em, column sep=2em,text height=1.5ex, text depth=0.25ex] 
{  X     &  &    X                     \\
 Y   &   &   Y                           \\} ;
\path[->,font=\scriptsize]  
(m-1-1) edge node[above] {$\tilde{f} $} (m-1-3)
(m-1-1) edge node[left] {$ q$} (m-2-1)
(m-1-3) edge node[right] {$ q$} (m-2-3)
(m-2-1) edge node[below] {$ f$} (m-2-3);
\end{tikzpicture}
\end{center} 
\end{proposition}
\begin{proof}
Let $s$ be a global section of $(\omega^{\vee}_Y)^{\otimes n}\cong \cO_Y(-nK_Y)$ having $D$ as its divisor of zeros. Then $X=\mathbb{W}_m(\omega^{\vee}_Y,s)$. It follows from the universal property described in Proposition \ref{universal} that we have a base change diagram
\begin{center}
\begin{tikzpicture}
[description/.style={fill=white,inner sep=2pt}]
\matrix (m) [matrix of math nodes, row sep=3em, column sep=2em,text height=1.5ex, text depth=0.25ex] 
{  
\mathbb{W}_n(f^*\omega^{\vee}_S,f^*s)        &  &     \mathbb{W}_n(\omega^{\vee}_S,s)                      \\
Y   &   &   Y                           \\} ;
\path[->,font=\scriptsize]  
(m-1-1) edge node[above] {$f' $} (m-1-3)
(m-1-1) edge node[left] {$ $} (m-2-1)
(m-1-3) edge node[right] {$ $} (m-2-3)
(m-2-1) edge node[below] {$ f$} (m-2-3);
\end{tikzpicture}
\end{center} 
Next observe that the cotangent morphism $f^*\Omega^1_Y\ra \Omega^1_Y$ induces an isomorphism $f^*\omega_Y\ra \omega_Y$. Dualizing we derive that there exists an isomorphism $\phi:\omega^{\vee}_Y\ra f^*\omega^{\vee}_Y$.
Since $\phi^{\otimes n}$ is an isomorphism, we derive that section $\phi^{\otimes n}(s)$ of $f^*(\omega_Y^{\vee})^{\otimes n}$ also has $D$ as the divisor of zeros. According to the fact that $f^*D=D$ we derive that $f^*s$ has $D$ as the divisor of zeros. Hence there exists $\alpha\in \CC^*$ such that $f^*s=\alpha \phi^{\otimes n}(s)=\phi^{\otimes n}(\alpha s)$.
Therefore, again by universal property of Proposition $\ref{universal}$ map $\phi$ induces an isomorphism $f'':\mathbb{W}_m(\omega^{\vee}_Y,\alpha s)\ra \mathbb{W}_m(f^*\omega^{\vee}_Y,f^*s)$ of schemes over $Y$. Finally, since $\alpha s$ and $s$ have the same divisor of zeros, there exists an isomorphism $f''':\mathbb{W}_m(\omega_Y^{\vee},s)\ra \mathbb{W}_m(\omega^{\vee}_Y,\alpha s)$ of schemes over $Y$. Now the composition $\widetilde{f}=f'\cdot f''\cdot f'''$ is a lift of $f$.
\end{proof}

\section{Construction of $X_4$}\label{3}

In this section we construct $X_4$ explicitly as a double covering of some rational surface.\\ Consider four points  $P=\{p_1,p_2,p_3,p_4\}$ of $\PP^2$ such that no three of them are on the same line. Blow them up to get a Del Pezzo surface $S_5=\mathrm{Bl}_{P}(\PP^2)=\mathrm{Bl}_{p_1,p_2,p_3,p_4}(\PP^2)$ of degree $5$. For $1\leq i<j\leq 4$ denote by $E_{ij}$ the strict transform on $\mathrm{Bl}_{P}(\PP^2)$ of a line on $\PP^2$ passing through points $P\setminus \{p_i,p_j\}=\{p_1,p_2,p_3,p_4\}\setminus \{p_i,p_j\}$. For $1\leq i\leq 4$ denote by $E_{0i}$ the exceptional divisor of $\mathrm{Bl}_{P}(\PP^2)$ over $p_i$. 

\begin{proposition}[{\cite[Section 8.5.1]{DIC}}]\label{delpezzo}
The following assertions hold.
\begin{enumerate}[label=\emph{\textbf{(\arabic*)}}, leftmargin=*]
\item  Curves $E_{ij}$ for $0\leq i<j\leq 4$ are all irreducible $(-1)$-curves on $S_5$.
\item  The divisor $E=\sum_{0\leq i<j\leq 4}E_{ij}$ is linearly equivalent to $-2K_{S_5}$.
\item  The incidence graph of curves $\{E_{ij}\}_{0\leq i<j\leq 4}$ is the Petersen graph in the \ref{fig1}.
\begin{equation}\label{fig1}
\begin{tikzpicture}[scale=0.25, every node/.style={scale=0.8}]
\draw (-4.7,2.4) node[scale=0.8]  {$ (02) $};
\draw (4.7,2.4) node[scale=0.8]  {$ (13) $};
\draw (0,9.8) node[scale=0.8]  {$ (01)$};
\draw (-9.9,3) node[scale=0.8]  {$ (34) $};
\draw (9.9,3) node[scale=0.8]  {$ (24) $};
\draw (-1.5,4.9)node[scale=0.8]  {$ (23) $};
\draw (4.3,-4)node[scale=0.8]  {$ (14) $};
\draw (-4.3,-4)node[scale=0.8]  {$ (04) $};
\draw (-6.5,-7.8)node[scale=0.8]  {$ (12) $};
\draw (6.5,-7.8)node[scale=0.8]  {$ (03) $};

\fill [black] (90 :5) circle (0.35cm);
\fill [black] (306 :5) circle (0.35cm);
\fill [black] (18 :5) circle (0.35cm);
\fill [black] (162 :9) circle (0.35cm);
\fill [black] (234 :9) circle (0.35cm);
\fill [black] (18 :9) circle (0.35cm);

\draw [black,very thick] (90:9) to (162:9);
\draw [black, very thick] (162:9) to (234:9);
\draw [black, line width=1mm] (234:9) to (306:9);
\draw [black, very thick] (306:9) to (18:9);
\draw [black, very thick] (18:9) to (90:9);

\draw [black, line width=1mm] (90:9) to (90:5);
\draw [black,very thick] (162:9) to (162:5);
\draw [black, very thick] (234:9) to (234:5);
\draw [black, very thick] (306:9) to (306:5);
\draw [black, very thick] (18:9) to (18:5);

\draw [black,very thick] (90:5) to (234:5);
\draw [black,  very thick] (90:5) to (306:5);
\draw [black,  line width=1mm] (162:5) to (18:5);
\draw [black, very thick] (18:5) to (234:5);
\draw [black, very thick] (306:5) to (162:5);

\fill [black] (162 :5) circle (0.35cm);
\fill [black] (234 :5) circle (0.35cm);
\fill [black] (90 :9) circle (0.35cm);
\fill [black] (306 :9) circle (0.35cm);
\end{tikzpicture}\tag{Figure 1}
\end{equation}
Three thick edges in the \ref{fig1} describe three linearly equivalent divisors whose complete linear system defines a fibration $S_5\ra \PP^1$ with general fiber being smooth and rational curve. It is usually called a conic fibration due to the fact that its fibers are of degree two with respect to $-K_{S_5}$.
\end{enumerate}
\end{proposition}
\vspace{0.5cm}
\noindent
Let $m:Y_4\ra S_5$ be the blowing up of intersection points of curves $\{E_{ij}\}_{0\leq i<j\leq 4}$. Let $F_{ij}$ be a strict transform of $E_{ij}$ in $Y_4$ for every $0\leq i<j\leq 4$. For any $\{i,j\}$, $\{k,l\}\subseteq \{0,1,2,3,4\}$ such that $\{i,j\}\cap \{k,l\}=\emptyset$ we denote by $F_{(ij)(kl)}$ a curve on $Y_4$ which is the preimage of the intersection point $E_{ij}\cap E_{kl}$. Next result is a consequence of the Proposition \ref{delpezzo}.

\begin{proposition}\label{petersen}
The following statements hold.
\begin{enumerate}[label=\emph{\textbf{(\arabic*)}}, leftmargin=*]
\item $F_{ij}$ are pairwise disjoint smooth rational $(-4)$-curves on $Y_4$ for $0\leq i<j\leq 4$.
\item The divisor $\sum_{0\leq i<j\leq 4}F_{ij}$  is linearly equivalent to $-2K_Y$.
\item The incidence graph of curves $F_{ij}$ and $F_{(ij)(kl)}$ for $0\leq i<j\leq 4$, $0\leq k<l\leq 4$ and $\{i,j\}\cap \{k,l\}=\emptyset$ is the extended Petersen graph in the \ref{fig2}.
\begin{equation}\label{fig2}
\begin{tikzpicture}[scale=0.37, every node/.style={scale=0.8}]
\draw (-4.7,2) node[scale=0.5]  {$(02)$};
\draw (4.7,2) node[scale=0.5]   {$(13)$};
\draw (0,9.5) node[scale=0.5]   {$(01)$};
\draw (-9.2,3) node[scale=0.5]   {$(34)$};
\draw (9.2,3) node[scale=0.5]   {$(24)$};
\draw (-0.6,4.9)node[scale=0.5]   {$(23)$};
\draw (3.6,-4)node[scale=0.5]   {$(14)$};
\draw (-3.6,-4)node[scale=0.5]   {$(04)$};
\draw (-5.8,-7.6)node[scale=0.5]   {$(12)$};
\draw (5.8,-7.6)node[scale=0.5]   {$(03)$};

\draw (0,2.1) node[scale=0.5]  {$(02)(13)$};
\draw (-2.5,0.65) node[scale=0.5]  {$(23)(04)$};
\draw (2.5,0.65) node[scale=0.5]  {$(23)(14)$};
\draw (1.4,-1.8) node[scale=0.5]  {$(13)(04)$};
\draw (-1.4,-1.8) node[scale=0.5]  {$(02)(14)$};

\draw (-1,7) node[scale=0.5]  {$ (01)(23)$};
\draw (-6.3,2.6) node[scale=0.5]  {$ (34)(02)$};
\draw (6.3,2.6) node[scale=0.5]  {$ (24)(13) $};
\draw (-3.1,-5.8) node[scale=0.5]  {$ (04)(12)$};
\draw (3.1,-5.8) node[scale=0.5]  {$ (14)(03) $};

\draw (-5.4,6) node[scale=0.5]  {$ (01)(34) $};
\draw (5.4,6) node[scale=0.5]  {$ (01)(24) $};
\draw (-8,-2.4) node[scale=0.5]  {$ (34)(12) $};
\draw (8.05,-2.4) node[scale=0.5]  {$ (24)(03) $};
\draw (0,-7.9) node[scale=0.5]  {$ (12)(03) $};

\fill [black] (90 :9) circle (0.2cm);
\fill [black] (162 :9) circle (0.2cm);
\fill [black] (234 :9) circle (0.2cm);
\fill [black] (306 :9) circle (0.2cm);
\fill [black] (18 :9) circle (0.2cm);

\draw [black ,very thick] (90:9) to (54:7.4);
\draw [black, very thick] (90:9) to (126:7.4);
\draw [black,very thick] (162:9) to (126:7.4);
\draw [black,very thick] (162:9) to (198:7.4);
\draw [black,very thick] (234:9) to (198:7.4);
\draw [black,very thick] (234:9) to (270:7.4);
\draw [black,very thick] (306:9) to (270:7.4);
\draw [black,very thick] (306:9) to (342:7.4);
\draw [black,very thick] (18:9) to (342:7.4);
\draw [black,very thick] (18:9) to (54:7.4);

\fill [black] (54 :7.4) circle (0.2cm);
\fill [black] (126 :7.4) circle (0.2cm);
\fill [black] (198 :7.4) circle (0.2cm);
\fill [black] (270 :7.4) circle (0.2cm);
\fill [black] (342 :7.4) circle (0.2cm);

\fill [black] (90 :7) circle (0.2cm);
\fill [black] (162 :7) circle (0.2cm);
\fill [black] (234 :7) circle (0.2cm);
\fill [black] (306 :7) circle (0.2cm);
\fill [black] (18 :7) circle (0.2cm);

\draw [black,very thick] (90:9) to (90:7);
\draw [black,very thick] (162:9) to (162:7);
\draw [black,very thick] (234:9) to (234:7);
\draw [black,very thick] (306:9) to (306:7);
\draw [black,very thick] (18:9) to (18:7);

\draw [black,very thick] (90:7) to (90:5);
\draw [black,very thick] (162:7) to (162:5);
\draw [black,very thick] (234:7) to (234:5);
\draw [black, very thick] (306:7) to (306:5);
\draw [black,very thick] (18:7) to (18:5);

\fill [black] (90 :5) circle (0.2cm);
\fill [black] (162 :5) circle (0.2cm);
\fill [black] (234 :5) circle (0.2cm);
\fill [black] (306 :5) circle (0.2cm);
\fill [black] (18 :5) circle (0.2cm);

\draw [black,very thick] (90:5) to (162:1.6);
\draw [black,very thick] (90:5) to (18:1.6);
\draw [black,very thick] (234:5) to (162:1.6);
\draw [black,very thick] (234:5) to (306:1.6);
\draw [black,very thick] (18:5) to (306:1.6);
\draw [black,very thick] (18:5) to (90:1.6);
\draw [black,very thick] (162:5) to (90:1.6);
\draw [black,very thick] (162:5) to (234:1.6);
\draw [black,very thick] (306:5) to (234:1.6);
\draw [black,very thick] (306:5) to (18:1.6);

\fill [black] (90 :1.6) circle (0.2cm);
\fill [black] (162 :1.6) circle (0.2cm);
\fill [black] (234 :1.6) circle (0.2cm);
\fill [black] (306 :1.6) circle (0.2cm);
\fill [black] (18 :1.6) circle (0.2cm);
\end{tikzpicture}\tag{Figure 2}
\end{equation}
\end{enumerate}
\end{proposition}

\begin{definition}
We define $X_4$ to be the anticanonical degree two covering of $Y_4$ branched along $B_{\pi}=\sum_{0\leq i<j\leq 4}F_{ij}$. We denote by $\pi:X_4\ra Y_4$ the anticanonical covering map and by $\sigma:X_4\ra X_4$ the unique automorphism of $\pi$ having order two. 
\end{definition}
\noindent
Since $F_{ij}$ are branch curves of $\pi$, we derive that $\pi^*F_{ij}=2L_{ij}$ for some $(-2)$-curve $L_{ij}$ on $X_4$ for $0\leq i<j\leq 4$. These curves are pairwise disjoint and $\coprod_{0\leq i<j\leq 4}L_{ij}$ is a fixed locus of $\sigma$ on $X_4$.

\begin{proposition}\label{linearization}
$X_4$ is a K3 surface with $\rho(X_4)=20$. The automorphism $\sigma$ acts as identity on the Picard group of $X_4$ and every line bundle $\cL$ on $X_4$ is linearizable with respect to the group $\{1_{X_4},\sigma\}$.
\end{proposition}
\begin{proof}
Note that the branch divisor of $\pi$ is smooth. Hence $X_4$ is a smooth surface. Clearly it is projective as a finite covering of a projective surface $Y_4$.\\
First by Riemann-Hurwitz formula and Proposition \ref{petersen} we derive that
$$\Omega^2_{X_4}=\pi^*\Omega^2_{Y_4}\otimes_{\cO_{X_4}}\cO_{X_4}\left(\sum_{0\leq i<j\leq 4}L_{ij}\right)\cong  \pi^*\cO_{Y_4}(K_{Y_4})\otimes_{\cO_{X_4}}\pi^*\cO_{Y_4}\left(\frac{1}{2}\sum_{0\leq i<j\leq 4}F_{ij}\right)\cong \pi^*\cO_{Y_4}\cong \cO_{X_4}$$
Next $\pi_*\cO_{X_4}=\cO_{Y_4}\oplus \cO_{Y_4}(K_{Y_4})$ and $\mathrm{H}^1(Y_4,\cO_{Y_4}(K_{Y_4}))=\mathrm{H}^1(Y_4,\cO_{Y_4})=0$ by Serre duality. Hence $\mathrm{H}^1(X_4,\pi_*\cO_{X_4})=0$. Now by Leray spectral sequence we have $\mathrm{H}^1(X_4,\cO_{X_4})\cong \mathrm{H}^1(\pi_*\cO_{X_4})=0$. Therefore, $X_4$ is a K3 surface.\\
Note that $\mathrm{rank}_{\ZZ}(\mathrm{Cl}(Y_4))=20$ and we have morphisms of abelian groups $\pi^*:\mathrm{Cl}(Y_4)\ra \mathrm{Cl}(X_4)$ and $\pi_*:\mathrm{Cl}(X_4)\ra \mathrm{Cl}(Y_4)$
such that $\pi_*\pi^*=\mathrm{deg}(\pi)1_{\mathrm{Cl}(Y_4)}=2\cdot 1_{\mathrm{Cl}(Y_4)}$. Thus $\pi^*$ is a monomorphism. Since $Y_4$ and $X_4$ are smooth, we derive that $\pi^*(\mathrm{Pic}(Y_4))=\pi^*(\mathrm{Cl}(Y_4))\subseteq \mathrm{Cl}(X_4)=\mathrm{Pic}(X_4)$ is a subgroup of rank $20$. Thus $\rho(X_4)=20$.\\
For every line bundle $\cL$ on $Y_4$ we have
$$\sigma^*(\pi^*\cL)=(\pi\sigma)^*(\cL)=\pi^*\cL$$
Thus $\sigma$ acts as identity on $\pi^*(\mathrm{Pic}(Y_4))$. Since this is a subgroup of maximal rank in the torsion free group $\mathrm{Pic}(X_4)$, we derive that $\sigma^*=1_{\mathrm{Pic}(X_4)}$.\\
Denote by $\left<\sigma \right>$ the subgroup of $\mathrm{Aut}(X_4)$ generated by $\sigma$.  Let $\mathrm{Pic}^{\left<\sigma \right>}(X_4)$ be the group  of $\left<\sigma \right>$-linearized line bundles on $X_4$ and let $\mathrm{Pic}(X_4)^{\left<\sigma \right>}$ be the group of line bundles on $X_4$ invariant with respect to the action of $\left<\sigma \right>$. Since $\sigma^*=1_{\mathrm{Pic}(X_4)}$, we derive that $\mathrm{Pic}(X_4)^{\left<\sigma \right>}=\mathrm{Pic}(X_4)$. Now according to \cite[Remark 7.2]{DIN} there is an exact sequence
\begin{center}
\begin{tikzpicture}
[description/.style={fill=white,inner sep=2pt}]
\matrix (m) [matrix of math nodes, row sep=3em, column sep=2em,text height=1.5ex, text depth=0.25ex] 
{ 0  & \Hom(\left<\sigma \right>,\CC^*)&    \mathrm{Pic}^{\left<\sigma \right>}(X_4) &\mathrm{Pic}(X_4)^{\left<\sigma \right>}&\mathrm{H}^2(\left<\sigma \right>,\CC^*)           \\} ;
\path[->,font=\scriptsize]  
(m-1-1) edge node[auto] {$  $} (m-1-2)
(m-1-2) edge node[auto] {$ $} (m-1-3)
(m-1-3) edge node[auto] {$ $} (m-1-4)
(m-1-4) edge node[auto] {$ $} (m-1-5);
\end{tikzpicture}
\end{center}
where the arrow $\mathrm{Pic}^{\left<\sigma \right>}(X_4)\ra \mathrm{Pic}(X_4)^{\left<\sigma \right>}=\mathrm{Pic}(X_4)$ forgets about the $\left<\sigma \right>$-linearization and $\CC^*$ is a trivial $\left<\sigma \right>$-module. According to $\left<\sigma \right>\cong \ZZ/2\ZZ$ and \cite[Theorem 6.2.2]{WEIBEL} we derive that if $\CC^*$ admits trivial action of $\ZZ/2\ZZ$, then $\mathrm{H}^2(\ZZ/2\ZZ,\CC^*)=0$. This proves that arrow $\mathrm{Pic}^{\left<\sigma \right>}(X_4)\ra \mathrm{Pic}(X_4)^{\left<\sigma \right>}=\mathrm{Pic}(X_4)$ is surjective and hence every line bundle on $X_4$ admits a $\left<\sigma \right>$-linearization.
\end{proof}

\begin{lemma}\label{lemma}
Let $F$ be a curve on $Y_4$ with negative self intersection. Then either $\pi^*F=2L$ or $\pi^*F=L$ for some $(-2)$-curve $L$ on $X_4$. The first case holds if and only if $F$ is a branch curve of $\pi$.
\end{lemma}
\begin{proof}
There are three possibilities.
\begin{enumerate}[label=\textbf{(\arabic*)}, leftmargin=*]
\item $\pi^*F=2L$ where $L$ is a curve on $X_4$ and the map $L\ra F$ induced by $\pi$ is birational.
\item $\pi^*F=L$ where $L$ is a curve on $X_4$ and the map $L\ra F$ induced by $\pi$ is of degree two.
\item $\pi^*F=L_1+L_2$ where $L_1$, $L_2$ are distinct curves on $X_4$.
\end{enumerate}
Note that in \textbf{(3)} we have $L_1^2+L_2^2+2L_1.L_2=(\pi^*F)^2<0$. Observe that $\sigma^*L_1=L_2$ as $\sigma$ acts transitively on fibers of $\pi$. According to the fact that $\sigma^*=1_{\mathrm{Pic}(X_4)}$, we derive that $L_1\sim L_2$ and hence $L_1^2+L_2^2+2L_1.L_2=4L_1.L_2>0$ a contradiction. So the only possibilities remaining are \textbf{(1)} and \textbf{(2)}. In both cases $L^2<0$ and since $X_4$ is a K3 surface, we derive that $L$ is a $(-2)$-curve on $X_4$.\\
Moreover, \textbf{(1)} holds if the ramification index of $\pi$ at the generic point of $L$ is $2$ and this implies that $F$ is a branch curve of $\pi$.
\end{proof}
\noindent
For $\{i,j\}$, $\{k,l\}\subseteq \{0,1,2,3,4\}$ and $\{i,j\}\cap \{k,l\}=\emptyset$ we define $L_{(ij)(kl)}=\pi^*F_{(ij)(kl)}$. 

\begin{corollary}\label{petersencor}
Every curve $L_{(ij)(kl)}$ is a smooth rational curve on $X_4$.\\
The incidence graph of curves  $L_{ij}$ for $0\leq i<j\leq 4$ and $L_{(ij)(kl)}$ for $\{i,j\}$, $\{k,l\}\subseteq \{0,1,2,3,4\}$ and $\{i,j\}\cap \{k,l\}=\emptyset$ is the extended Petersen graph.
\end{corollary}
\begin{proof} 
The first assertion is a direct consequence of Lemma \ref{lemma} and the fact that $F_{(ij)(kl)}^2=-1$. The second assertion follows from Proposition \ref{petersen} and the fact that $\pi^*$ preserves intersection pairing up to multiplication by degree of $\pi$. 
\end{proof}

\begin{tabel}\label{table}
The following table collects information about all curves defined so far.
\begin{center}
\begin{tabular}{ | p{3cm} | p{3.5cm} | p{3.5cm} | p{3.5cm} |}
\hline
\vspace{0.5pt}$S_5$ (Prop. \ref{delpezzo})& \vspace{0.5pt}$E_{ij}$ for $0\leq i<j\leq 4$. These are all $(-1)$-curves on $S_5$. &   & The Petersen graph \ref{fig1} describes their intersection. \\ \hline
\vspace{0.5pt}$Y_4$ (Prop. \ref{petersen})& \vspace{0.5pt} The strict transform $F_{ij}$ of $E_{ij}$ for $0\leq i<j\leq 4$  . These are smooth, rational $(-4)$-curves.  & \vspace{0.5pt} Blow up $F_{(ij)(kl)}$  of intersection point $E_{ij}\cap E_{kl}$ for $\{i,j\}$, $\{k,l\}\subseteq \{0,1,2,3,4\}$ and $\{i,j\}\cap \{k,l\}=\emptyset$. These are $(-1)$-curves.  & The extended Petersen graph \ref{fig2} describes their intersection. \\ \hline
\vspace{0.5pt} $X_4$ (Cor. \ref{petersencor})  & \vspace{0.5pt} $L_{ij}=\frac{1}{2}\pi^*F_{ij}$  for $0\leq i<j\leq 4$. These are $(-2)$-curves. & \vspace{0.5pt}  $L_{(ij)(kl)}=\pi^*F_{(ij)(kl)}$ for $\{i,j\}$, $\{k,l\}\subseteq \{0,1,2,3,4\}$ and $\{i,j\}\cap \{k,l\}=\emptyset$. These are $(-2)$-curves.& The extended Petersen graph \ref{fig2} describes their intersection. \\ \hline
\end{tabular}
\end{center}
\end{tabel}
\hspace{1cm}

\begin{proposition}\label{transc}
The cup product pairing restricted to the transcendental lattice $T_{X_4}$ is given by the matrix
$$\left[ \begin{array}{cc}
2 &0 \\
0& 2 \end{array} \right]$$
with respect to some basis of $T_{X_4}$.
\end{proposition}
\begin{proof}
Let us come back to the Del Pezzo surface $S_5=\mathrm{Bl}_P(\PP^2)$. The following divisors:
$$E_{01}+E_{23},\,E_{02}+E_{13},\,E_{12}+E_{03}$$
are linearly equivalent. Their linear system is base point free and its members consist of degree two divisors with respect to very ample divisor $-K_{S_5}$ and gives rise to the morphism $f:S_5 \ra \PP^1$ whose fibers are curves of degree two with respect to $-K_{S_5}$. This conic bundle has exactly three singular fibers given by the three divisors listed above. The morphism $p=f\cdot m\cdot \pi$ is an elliptic fibration. Recall that by Proposition $\ref{petersencor}$ the incidence graph of curves $L_{ij}$ and $L_{(ij)(kl)}$ for $0\leq i<j\leq 4$, $0\leq k<l\leq 4$ and $\{i,j\}\cap \{k,l\}=\emptyset$ is the extended Petersen graph. There are exactly three singular fibers of $p$ given by pulling back along $m\pi$ singular fibers of $f$. They are depicted in the \ref{fig3}.
\begin{equation}\label{fig3}
\begin{tikzpicture}[scale=0.37, every node/.style={scale=0.8}]
\draw [myred, very thick, line width=1mm] (90:9) to (54:7.4);
\draw [myred, very thick, line width=1mm] (90:9) to (126:7.4);
\draw [black, very thick] (162:9) to (126:7.4);
\draw [black, very thick] (162:9) to (198:7.4);
\draw [mygreen, very thick, line width=1mm] (234:9) to (198:7.4);
\draw [mygreen, very thick, line width=1mm] (234:9) to (270:7.4);
\draw [mygreen, very thick, line width=1mm] (306:9) to (270:7.4);
\draw [mygreen,very thick, line width=1mm] (306:9) to (342:7.4);
\draw [black, very thick] (18:9) to (342:7.4);
\draw [black, very thick] (18:9) to (54:7.4);

\draw [myred, very thick, line width=1mm] (90:9) to (90:7);
\draw [black, very thick] (162:9) to (162:7);
\draw [mygreen, very thick, line width=1mm] (234:9) to (234:7);
\draw [mygreen, very thick, line width=1mm] (306:9) to (306:7);
\draw [black,very thick] (18:9) to (18:7);

\draw [myred, very thick, line width=1mm] (90:7) to (90:5);
\draw [myblue, very thick, line width=1mm] (162:7) to (162:5);
\draw [black, very thick] (234:7) to (234:5);
\draw [black, very thick] (306:7) to (306:5);
\draw [myblue,very thick, line width=1mm] (18:7) to (18:5);

\draw [myred,very thick, line width=1mm] (90:5) to (162:1.6);
\draw [myred, very thick, line width=1mm] (90:5) to (18:1.6);
\draw [black,very thick] (234:5) to (162:1.6);
\draw [black, very thick] (234:5) to (306:1.6);
\draw [myblue, very thick, line width=1mm] (18:5) to (306:1.6);
\draw [myblue, very thick, line width=1mm] (18:5) to (90:1.6);
\draw [myblue, very thick, line width=1mm] (162:5) to (90:1.6);
\draw [myblue, very thick, line width=1mm] (162:5) to (234:1.6);
\draw [black, very thick] (306:5) to (234:1.6);
\draw [black, very thick] (306:5) to (18:1.6);

\fill [myred] (90 :9) circle (0.2cm);
\fill [black] (162 :9) circle (0.2cm);
\fill [mygreen] (234 :9) circle (0.2cm);
\fill [mygreen] (306 :9) circle (0.2cm);
\fill [black] (18 :9) circle (0.2cm);

\fill [myred] (54 :7.4) circle (0.2cm);
\fill [myred] (126 :7.4) circle (0.2cm);
\fill [mygreen] (198 :7.4) circle (0.2cm);
\fill [mygreen] (270 :7.4) circle (0.2cm);
\fill [mygreen] (342 :7.4) circle (0.2cm);

\fill [myred] (90 :7) circle (0.2cm);
\fill [myblue] (162 :7) circle (0.2cm);
\fill [mygreen] (234 :7) circle (0.2cm);
\fill [mygreen] (306 :7) circle (0.2cm);
\fill [myblue] (18 :7) circle (0.2cm);

\fill [myred] (90 :5) circle (0.2cm);
\fill [myblue] (162 :5) circle (0.2cm);
\fill [black] (234 :5) circle (0.2cm);
\fill [black] (306 :5) circle (0.2cm);
\fill [myblue] (18 :5) circle (0.2cm);

\fill [myblue] (90 :1.6) circle (0.2cm);
\fill [myred] (162 :1.6) circle (0.2cm);
\fill [myblue] (234 :1.6) circle (0.2cm);
\fill [myblue] (306 :1.6) circle (0.2cm);
\fill [myred] (18 :1.6) circle (0.2cm);

\draw (-4.7,2) node[myblue,scale=0.6]  {$ 2 $};
\draw (4.7,2) node[myblue,scale=0.6]  {$2$};
\draw (0,9.5) node[myred,scale=0.6]  {$2$};
\draw (-8.9,2.9) node  {$s$};
\draw (9,2.9) node {$s$};
\draw (-0.4,5)node[myred,scale=0.6]  {$2$};
\draw (3.4,-4)node  {$s$};
\draw (-3.3,-4)node  {$s$};
\draw (-5.7,-7.4)node[mygreen,scale=0.6]  {$2$};
\draw (5.7,-7.4)node[mygreen,scale=0.6]  {$2$};

\draw (0,2) node[myblue,scale=0.6]  {$2$};
\draw (-1.9,0.5) node[myred,scale=0.6]  {$1$};
\draw (1.9,0.5) node[myred,scale=0.6]  {$1$};
\draw (1.3,-1.5) node[myblue,scale=0.6]  {$1$};
\draw (-1.3,-1.5) node[myblue,scale=0.6]  {$1$};

\draw (-0.4,7) node[myred,scale=0.6]  {$2$};
\draw (-6.7,2.6) node[myblue,scale=0.6]  {$1$};
\draw (6.7,2.6) node[myblue,scale=0.6]  {$1$};
\draw (-3.7,-5.7) node[mygreen,scale=0.6]  {$1$};
\draw (3.7,-5.7) node[mygreen,scale=0.6]  {$1$};

\draw (-4.8,6.2) node[myred,scale=0.6]  {$1$};
\draw (4.8,6.2) node[myred,scale=0.6]  {$1$};
\draw (-7.4,-2.4) node[mygreen,scale=0.6] {$1$};
\draw (7.4,-2.4) node[mygreen,scale=0.6]  {$1$};
\draw (0,-7.8) node[mygreen,scale=0.6]  {$2$};
\end{tikzpicture}\tag{Figure 3}
\end{equation}
Here colored and thick subgraphs correspond to three singular fibers. Each colored vertex is labeled by its multiplicity in the corresponding fiber. Moreover, there are four black vertices that are connected with each colored subgraph by precisely one edge. These vertices correspond to sections of the fibration and are labeled by letters "s". Hence $p$ has precisely three singular fibers each of type $\tilde{D}_6$ according to Kodaira classification \cite[Chapter V, Section 7]{CCS} and it has at least four distinct sections. Using  Theorem \ref{Shioda}, we deduce that
$$20=r(p)+2+3\cdot 6$$
where $r(p)$ is the torsion free rank of the Mordell-Weil group of $p$. Hence $r(p)=0$. Let $n(p)$ be the order of the Mordell-Weil group of $p$. Then $4\leq n(p)$ due to existence of four distinct sections of $p$. Again using Theorem  \ref{Shioda}, we derive that
$$|\mathrm{det}(T_{X_4})|=\frac{4^3}{n(p)^2}\leq \frac{4^3}{4^2}=4$$
and $|\mathrm{det}(T_{X_4})|$ is a divisor of $4^3$. Hence $|\mathrm{det}(T_{X_4})|=1,2,4$. Now the fact that $T_{X_4}$ is a rank two, even, positive definite integral lattice, implies that $\mathrm{det}(T_{X_4})=4$. There exists precisely one rank two, even, positive definite integral lattice with discriminant equal to $4$ and it has a basis in which intersection form has the matrix
$$\left[ \begin{array}{cc}
2 &0 \\
0 & 2 \end{array} \right]\qedhere$$
\end{proof}
\begin{remark}
Since $\sigma^*$ acts as the identity on the algebraic lattice $S_{X_4}$, this implies that it also acts as the identity on the discriminant group $D_{X_4}$. On the other hand $\sigma:X_4\ra X_4$ is a non-symplectic involution of $X_4$. Hence $\sigma^*$ induces multiplication by $-1$ on $T_{X_4}$. Thus it acts as multiplication by $-1$ on $D_{X_4}$. Therefore, we derive that $D_{X_4}$ is a direct sum of copies of $\ZZ/2\ZZ$. Now one can use Nikulin's formula in {\cite[Theorem 4.2.2]{NIK2}} (see also \cite{NIK}) to give an alternative proof of Proposition \ref{transc}.
\end{remark}

\begin{corollary}
The K3 surface $X_4$ constructed in this section is isomorphic to the unique K3 surface with maximal Picard rank and with discriminant equal to four.
\end{corollary}
\begin{proof}
This follows from Proposition \ref{transc} and \cite[Theorem 4]{SHIN}.
\end{proof}

\section{Elliptic fibrations on $X_4$}\label{4}

\begin{corollary}
Let $p:X_4\ra \PP^1$ be an elliptic fibration. We consider $X_4$ as a variety equipped with the action of $\left<\sigma\right>=\{1_{X_4},\sigma\}$. Then there exists an action of $\ZZ/2\ZZ$ on $\PP^1$ such that $p$ is an equivariant morphism.
\end{corollary}
\begin{proof}
Let $\cL=p^*\cO_{\PP^1}(1)$. Since $p$ is a morphism with connected fibers we derive that $\mathrm{H}^0(X_4,\cL)$ is isomorphic to $\mathrm{H}^0(\PP^1,\cO_{\PP^1}(1))$. According to Proposition \ref{linearization} the line bundle $\cL$ admits a $\left<\sigma\right>$-linearization. Thus there exists a linear action of $\left<\sigma\right>$ on global sections of $\cL$ and there is a morphism $\mathrm{H}^0(X_4,\cL)\otimes_{\CC}\cO_{X_4}\ra \cL$ of sheaves with $\left<\sigma\right>$-linearizations. Since $\PP^1$ is the projectivization $\PP\left(\mathrm{H}^0(X_4,\cL)\right)$, we deduce that there exists an action of $\ZZ/2\ZZ$ on $\PP^1$ such that $p$ is equivariant.
\end{proof}
\noindent
Let $p:X_4\ra \PP^1$ be an elliptic fibration and $\tau:\PP^1\ra \PP^1$ the, possibly trivial, involution inducing the $\ZZ/2\ZZ$ action on $\PP^1$ which makes $p$ equivariant. Then there exists a fibration $q:Y_4\ra \PP^1$ such that the following diagram is commutative
\begin{center}
\begin{tikzpicture}
[description/.style={fill=white,inner sep=2pt}]
\matrix (m) [matrix of math nodes, row sep=3em, column sep=2em,text height=1.5ex, text depth=0.25ex] 
{ X_4 &  &  Y_4                        \\
  \PP^1 & &   \PP^1             \\} ;
\path[->,font=\scriptsize]  
(m-1-1) edge node[above] {$\pi $} (m-1-3)
(m-2-1) edge node[below]{$r$} (m-2-3)
(m-1-3) edge node[right] {$q$} (m-2-3)
(m-1-1) edge node[left] {$ p$} (m-2-1);
\end{tikzpicture}
\end{center} 
where $r:\PP^1\ra \PP^1$ is the quotient morphism with respect to the action of $\tau$ on $\PP^1$. We will call $q$ the fibration induced by $p$. Clearly $q$ is either an elliptic fibration or a fibration with general fiber being smooth, rational curve.

\begin{proposition}\label{ellipticX_4}
Let $p  :X_4\ra \PP^1$ be an elliptic fibration and $q:Y_4\ra \PP^1$ the induced fibration. Then the following hold.
\begin{enumerate}[label=\emph{\textbf{(\arabic*)}}, leftmargin=*]
\item $\tau=1_{\PP^1}$ if and only if $q$ is a fibration with general fiber being smooth rational curve. In this case every section of $p$ is a ramification curve of $\pi$.
\item $\tau\neq 1_{\PP^1}$ if and only if $q$ is an elliptic fibration. In this case there are at most two reducible fibers of $p$ and they are preimages of fixed points of $\tau$.
\end{enumerate}
\end{proposition}
\begin{proof}
According to Proposition \ref{transc} and Proposition \ref{elliptics} we deduce that every elliptic fibration on $X_4$ admits a section.\\
Suppose that the action of $\ZZ/2\ZZ$ on $\PP^1$ is trivial. In particular, we have 
$$p\cdot \sigma=q\cdot \pi\cdot \sigma=q\cdot \pi=p$$
We derive that $p(\sigma(F_{u}))=p(F_{u})$ where $u\in \PP^1$ is a point and $F_u=p^{-1}(u)$ is a fiber. This implies that $\sigma(F_u)=F_{u}$. Consider now a section $s:\PP^1\ra X_4$ of $p$. If $v=s(u)$, then $\{v\}=F_{u}\cap s(\PP^1)$. Hence $\sigma(v)\in \sigma(F_{u})=F_{u}$. On the other hand $s(\PP^1)$ is a smooth, rational curve on a K3 hence a $(-2)$-curve. Since $\sigma^*=1_{\mathrm{Pic}(X_4)}$, we derive that $\sigma(s(\PP^1))=s(\PP^1)$. According to $v\in s(\PP^1)$, we deduce that $\sigma(v)\in \sigma(s(\PP^1))=s(\PP^1)$. Thus $\sigma(v)\in F_{u}\cap s(\PP^1)=\{v\}$ and hence $v$ is a fixed point of $\sigma$. Therefore, every point of $s(\PP^1)$ is a fixed point of the action of $\sigma$. This implies that $s(\PP^1)$ is a ramification curve of $\pi$. Next note that $F_{u}=p^{-1}(u)$ and $C_{u}=q^{-1}(r(u))=q^{-1}(u)$ are smooth curves if one chooses $u\in \PP^1$ to be sufficiently general. Hence the morphism $F_{u}\ra C_{u}$ induced by $\pi$ is a ramified morphism of smooth curves. Thus we have $g(C_{u})<g(F_{u})=1$. Hence $C_{u}$ is a smooth, rational curve. This implies that $q$ is a fibration with general fiber being smooth, rational curve.\\
Suppose now that the action of $\ZZ/2\ZZ$ on $\PP^1$ is nontrivial. Then it is given by some nontrivial involution $\tau$ of $\PP^1$. Pick $u\in \PP^1$ such that $F_u=p^{-1}(u)$ is smooth and $\tau(u)\neq u$. We deduce that $F_{\tau(u)}=\sigma(F_u)$ is smooth. Moreover, we have 
$$\pi(F_u)=C_{r(u)}=\pi(F_{\tau(u)})$$
where $C_{r(u)}=q^{-1}(r(u))$ is some curve on $Y_4$. It is clear that if one chooses $u$ to be sufficiently general, then one may assume that $C_{r(u)}$ is smooth. Let $e_u$, $e_{\tau(u)}$ and $f_u$, $f_{\tau(u)}$ be ramification indexes and inertia indexes of $\pi$ at generic points of $F_{u}$ and $F_{\tau(u)}$. Then we have formula \cite[7.4.2, Formula 4.8]{LIU}
$$f_ue_u+f_{\tau(u)}e_{\tau(u)}=2$$
Hence $e_u=e_{\tau(u)}=f_u=f_{\tau(u)}=1$ and morphisms $F_u\ra C_{r(u)}$ and $F_{\tau(u)}\ra C_{r(u)}$ induced by $\pi$ are isomorphisms. This implies that $C_{r(u)}=q^{-1}(r(u))$ is a smooth elliptic curve. Hence $q$ is an elliptic fibration. Finally note that by Kodaira classification \cite[Chapter V, Section 7]{CCS} reducible fibers of $p$ correspond to unions of $(-2)$-curves on $X_4$. In particular, reducible fibers of $p$ are invariant under the action of $\sigma$. Hence every reducible fiber of $p$ is contracted by $p$ to a fixed point of $\tau$. According to the fact that $\tau$ is a nontrivial involution of $\PP^1$, it has two fixed points. Thus there are at most two reducible fibers of $p$.
\end{proof}
\noindent
Both types of elliptic fibrations described in the previous proposition are realized on $X_4$. For this observe that the elliptic fibration described in Proposition \ref{transc} is of the first type. Now consider divisors $D_1$ and $D_2$ corresponding to subgraphs of the configuration described in Proposition \ref{petersencor} and depicted by colored and thick parts in the \ref{fig4}.
\begin{equation}\label{fig4}
\begin{tikzpicture}[scale=0.37, every node/.style={scale=0.8}]
\draw (2.5,0.65) node[scale=0.5]  {$(23)(14)$};
\draw (1.4,-1.8) node[scale=0.5]  {$(13)(04)$};
\draw (-6.3,2.6) node[scale=0.5]  {$ (34)(02)$};
\draw (5.4,6) node[scale=0.5]  {$ (01)(24) $};
\draw (0,-7.9) node[scale=0.5]  {$ (12)(03) $};
\draw [black, very thick] (90:9) to (54:7.4);
\draw [black,very thick] (234:9) to (270:7.4);
\draw [black,very thick] (306:9) to (270:7.4);
\draw [black,very thick] (18:9) to (54:7.4);
\fill [black] (54 :7.4) circle (0.2cm);
\fill [black] (270 :7.4) circle (0.2cm);
\fill [black] (162 :7) circle (0.2cm);
\draw [black, very thick] (162:9) to (162:7);
\draw [black, very thick] (162:7) to (162:5);
\draw [black, very thick] (90:5) to (18:1.6);
\draw [black, very thick] (234:5) to (306:1.6);
\draw [black, very thick] (18:5) to (306:1.6);
\draw [black, very thick] (306:5) to (18:1.6);
\fill [black] (306 :1.6) circle (0.2cm);
\fill [black] (18 :1.6) circle (0.2cm);

\draw (8.05,-2.4) node[myblue, scale=0.5]  {$ (24)(03) $};
\draw (-4.7,2) node[myblue, scale=0.5]  {$(02)$};
\draw (4.7,2) node[myblue, scale=0.5]   {$(13)$};
\draw (9.2,3) node[myblue, scale=0.5]   {$(24)$};
\draw (3.6,-4)node[myblue, scale=0.5]   {$(14)$};
\draw (5.8,-7.6)node[myblue, scale=0.5]   {$(03)$};
\draw (0,2.1) node[myblue, scale=0.5]  {$(02)(13)$};
\draw (-1.4,-1.8) node[myblue, scale=0.5]  {$(02)(14)$};
\draw (6.3,2.6) node[myblue, scale=0.5]  {$ (24)(13) $};
\draw (3.1,-5.8) node[myblue, scale=0.5]  {$ (14)(03) $};
\fill [myblue] (306 :9) circle (0.2cm);
\fill [myblue] (18 :9) circle (0.2cm);
\draw [myblue,line width=1mm] (306:9) to (342:7.4);
\draw [myblue,line width=1mm] (18:9) to (342:7.4);
\fill [myblue] (342 :7.4) circle (0.2cm);
\fill [myblue] (306 :7) circle (0.2cm);
\fill [myblue] (18 :7) circle (0.2cm);
\draw [myblue,line width=1mm] (306:9) to (306:7);
\draw [myblue,line width=1mm] (18:9) to (18:7);
\draw [myblue, line width=1mm] (306:7) to (306:5);
\draw [myblue, line width=1mm] (18:7) to (18:5);
\fill [myblue] (162 :5) circle (0.2cm);
\fill [myblue] (306 :5) circle (0.2cm);
\fill [myblue] (18 :5) circle (0.2cm);
\draw [myblue,line width=1mm] (18:5) to (90:1.6);
\draw [myblue,line width=1mm] (162:5) to (90:1.6);
\draw [myblue,line width=1mm] (162:5) to (234:1.6);
\draw [myblue,line width=1mm] (306:5) to (234:1.6);
\fill [myblue] (90 :1.6) circle (0.2cm);
\fill [myblue] (234 :1.6) circle (0.2cm);

\draw (0,9.5) node[myred, scale=0.5]   {$(01)$};
\draw (-9.2,3) node[myred, scale=0.5]   {$(34)$};
\draw (-0.6,4.9)node[myred, scale=0.5]   {$(23)$};
\draw (-3.6,-4)node[myred, scale=0.5]   {$(04)$};
\draw (-5.8,-7.6)node[myred, scale=0.5]   {$(12)$};
\draw (-2.5,0.65) node[myred, scale=0.5]  {$(23)(04)$};
\draw (-1,7) node[myred, scale=0.5]  {$ (01)(23)$};
\draw (-3.1,-5.8) node[myred, scale=0.5]  {$ (04)(12)$};
\draw (-5.4,6) node[myred, scale=0.5]  {$ (01)(34) $};
\draw (-8,-2.4) node[myred, scale=0.5]  {$ (34)(12) $};
\fill [myred] (90 :9) circle (0.2cm);
\fill [myred] (162 :9) circle (0.2cm);
\fill [myred] (234 :9) circle (0.2cm);
\draw [myred, line width=1mm] (90:9) to (126:7.4);
\draw [myred,line width=1mm] (162:9) to (126:7.4);
\draw [myred, line width=1mm] (162:9) to (198:7.4);
\draw [myred, line width=1mm] (234:9) to (198:7.4);
\fill [myred] (126 :7.4) circle (0.2cm);
\fill [myred] (198 :7.4) circle (0.2cm);
\fill [myred] (90 :7) circle (0.2cm);
\fill [myred] (234 :7) circle (0.2cm);
\draw [myred, line width=1mm] (90:9) to (90:7);
\draw [myred, line width=1mm] (234:9) to (234:7);
\draw [myred, line width=1mm] (90:7) to (90:5);
\draw [myred, line width=1mm] (234:7) to (234:5);
\fill [myred] (90 :5) circle (0.2cm);
\fill [myred] (234 :5) circle (0.2cm);
\draw [myred, line width=1mm] (90:5) to (162:1.6);
\draw [myred, line width=1mm] (234:5) to (162:1.6);
\fill [myred] (162 :1.6) circle (0.2cm);
\end{tikzpicture}\tag{Figure 4}
\end{equation}
We have that $D_1.D_2=0$ and both these divisors are of Kodaira \cite[Chapter V, Section 7]{CCS} type $\tilde{A}_9$. According to Corollary \ref{ellipticK3} there exists an elliptic fibration $p:X_4\ra \PP^1$ such that two of its fibers are precisely divisors $D_1$ and $D_2$.  Since $D_1\cup D_2$ contains all ramification curves of $\pi$, we deduce that this elliptic fibration is of type \textbf{(2)} with respect to Proposition \ref{ellipticX_4}. 

\begin{corollary}
$Y_4$ is an elliptic rational surface.
\end{corollary}
\begin{proof}
According to Proposition \ref{ellipticX_4} fibration of type \textbf{(2)} induces an elliptic fibration on $Y_4$. Since by previous discussion fibrations of type \textbf{(2)} exist, we derive the assertion.
\end{proof}

\section{Automorphism group of $X_4$ and Cremona group}\label{5}
\noindent
In order to describe automorphism group of $X_4$ in geometric terms we need to prove first certain result on negative curves on $Y_4$ and $X_4$. Recall that $\pi:X_4\ra Y_4$ denotes the anticanonical covering of degree two.
\begin{proposition}\label{classes}
Let $F$ be a curve on $Y_4$ such that $F^2<0$. Then $F$ is a smooth rational curve and $F^2\in \{-1,-4\}$. Moreover, these curves fit into the following classes.
\begin{itemize}[leftmargin=2em]
\item[\emph{\textbf{($\bd{b}$)}}] $F^2=-4$ and $F$ is a branch curve of $\pi$ i.e $F=F_{ij}$ for some $0\leq i<j\leq 4$.
\item[\emph{\textbf{($\bd{f_1}$)}}]  $F^2=-1$ and $F$ intersects exactly one of the curves in the set $\{F_{ij}\}_{0\leq i<j\leq 4}$ at two distinct points and transversally
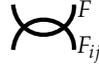
\begin{center}
\begin{tikzpicture}[scale=0.2]
\node[scale=0.9] at (2.7,1.7) {$F$};
\node[scale=0.9] at (3,-1) {$F_{ij}$};
\draw [ultra thick,black] (0,0) to[out=0,in=-85] (2,2);
\draw [ultra thick,black] (0,1) to[out=0,in=85] (2,-1);
\draw [ultra thick,black] (0,0) to[out=180,in=-85] (-2,2);
\draw [ultra thick,black] (0,1) to[out=180,in=85] (-2,-1);
\end{tikzpicture}
\end{center}
\item[\emph{\textbf{($\bd{f_2}$)}}] $F^2=-1$ and $F$ intersects exactly two curves in the set $\{F_{ij}\}_{0\leq i<j\leq 4}$ each at one point and transversally
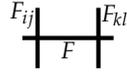
\begin{center}
\begin{tikzpicture}[scale=0.2]
\node[scale=0.9] at (0,-0.8) {$F$};
\node[scale=0.9] at (-3,1.5) {$F_{ij}$};
\node[scale=0.9] at (3.1,1.5) {$F_{kl}$};
\draw [black,ultra thick] (-3,0) to (3,0);
\draw [black, ultra thick] (-2,-2) to (-2,2);
\draw [black, ultra thick] (2,-2) to (2,2);
\end{tikzpicture}
\end{center}
\end{itemize}
Let $L$ be a curve on $X_4$ such that, $L^2<0$. Then $L$ is a smooth rational curve and $L^2=-2$.\\ Moreover, these curves fit into the following classes.
\begin{itemize}[leftmargin=2em]
\item[\emph{\textbf{($\bd{r}$)}}] $L$ is a ramification curve of $\pi$ i.e $L=L_{ij}$ for some $0\leq i<j\leq 4$.
\item[\emph{\textbf{($\bd{l_1}$)}}] $L$ intersects exactly one of the curves in the set $\{L_{ij}\}_{0\leq i<j\leq 4}$ at two distinct points and transversally
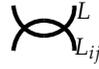
\begin{center}
\begin{tikzpicture}[scale=0.2]
\node[scale=0.9] at (2.6,1.7) {$L$};
\node[scale=0.9] at (2.9,-1) {$L_{ij}$};
\draw [ultra thick,black] (0,0) to[out=0,in=-85] (2,2);
\draw [ultra thick,black] (0,1) to[out=0,in=85] (2,-1);
\draw [ultra thick,black] (0,0) to[out=180,in=-85] (-2,2);
\draw [ultra thick,black] (0,1) to[out=180,in=85] (-2,-1);
\end{tikzpicture}
\end{center}
\item[\emph{\textbf{($\bd{l_2}$)}}] $L$ intersects exactly two curves in the set $\{L_{ij}\}_{0\leq i<j\leq 4}$ each at one point and transversally
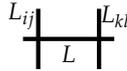
\begin{center}
\begin{tikzpicture}[scale=0.2]
\node[scale=0.9] at (0,-1) {$L$};
\node[scale=0.9] at (-3.1,1.5) {$L_{ij}$};
\node[scale=0.9] at (3.1,1.5) {$L_{kl}$};
\draw [black,ultra thick] (-3,0) to (3,0);
\draw [black, ultra thick] (-2,-2) to (-2,2);
\draw [black, ultra thick] (2,-2) to (2,2);
\end{tikzpicture}
\end{center}
\end{itemize}
We have bijections
$$\bd{b}\ni F\mapsto L=\frac{1}{2}\pi^*F\in \bd{r}$$
$$\bd{f_1}\ni F\mapsto L=\pi^*F\in \bd{l_1}$$
$$\bd{f_2}\ni F\mapsto L=\pi^*F\in \bd{l_2}$$
\end{proposition}
\begin{proof}
According to Lemma \ref{lemma} we have two cases:
\begin{enumerate}[label=\textbf{(\arabic*)}, leftmargin=*]
\item $\pi^*F=2L$ where $L$ is a curve on $X_4$ and the map $L\ra F$ induced by $\pi$ is birational.
\item $\pi^*F=L$ where $L$ is a curve on $X_4$ and the map $L\ra F$ induced by $\pi$ is of degree two.
\end{enumerate}
In case \textbf{(1)} we have $F^2=2L^2=-4$ and $F$ is a component of branch divisor $B_{\pi}$. Thus $F=F_{ij}$ for some $0\leq i<j\leq 4$. This gives the class \textbf{b} and the corresponding class \textbf{r}.\\
Let us study \textbf{(2)}. In this case we derive $F^2=\frac{1}{2}L^2=-1$. Since $\pi^*F=L$, the ramification index of a local parameter at generic point of $F$ is equal to $1$. This means that $F$ is not contained in branch locus of $\pi$ and hence $0\leq F.B_{\pi}=-2F.K_{Y_4}$. Hence $F.K_{Y_4}\leq 0$. By formula on arithmetic genus we have
$$2p_a(F)-2= F^2+F.K_{Y_4}\leq -1$$
Thus $p_a(F)=0$ and $F.K_{Y_4}=-1$. This implies that $F.B_{\pi}=2$. Moreover, morphism $L\ra F$  induced by $\pi$ is a degree two cyclic covering of $F$ branched along divisor ${B_{\pi}}_{\mid F}$. Since $L$ is smooth, we derive that ${B_{\pi}}_{\mid F}$ must be smooth. Hence ${B_{\pi}}_{\mid F}$ is a union of two points each with multiplicity one. This gives classes $\bd{f_1}$ and $\bd{f_2}$. Note that every $(-2)$-curve $L$ on $X_4$ is either the pullback or half the pullback of a negative curve on $Y_4$. This gives classes $\bd{l_1}$ and $\bd{l_2}$ corresponding to $\bd{f_1}$ and $\bd{f_2}$.
\end{proof}

\begin{remark}
All these classes are nonempty. For this note that class $\bd{r}$ is clearly nonempty. Moreover, $L_{(ij)(kl)}$ is a curve of type $\bd{l_2}$ that intersects $L_{ij}$ and $L_{kl}$ for disjoint $\{i,j\}$, $\{k,l\}\subseteq \{0,1,2,3,4\}$. We now prove that for $0\leq i<j\leq 4$ there exists a curve $L\in \bd{l_1}$ that intersects $L_{ij}$. Without loss of generality we may assume that $i=3$ and $j=4$. Consider the extended Petersen graph with some extra decorations as shown in the \ref{fig5}.
\begin{equation}\label{fig5}
\begin{tikzpicture}[scale=0.37, every node/.style={scale=0.8}]

\draw (2.5,0.65) node[scale=0.5]  {$(23)(14)$};
\draw (1.4,-1.8) node[scale=0.5]  {$(13)(04)$};
\draw (-6.3,2.6) node[scale=0.5]  {$ (34)(02)$};
\draw (-5.4,6) node[scale=0.5]  {$ (01)(34) $};
\draw (-8,-2.4) node[scale=0.5]  {$ (34)(12) $};
\draw (8.05,-2.4) node[scale=0.5]  {$ (24)(03) $};

\draw [black, very thick] (90:9) to (126:7.4);
\draw [black,very thick] (162:9) to (126:7.4);
\draw [black,very thick] (162:9) to (198:7.4);
\draw [black,very thick] (234:9) to (198:7.4);

\fill [black] (126 :7.4) circle (0.2cm);
\fill [black] (198 :7.4) circle (0.2cm);
\fill [black] (162 :7) circle (0.2cm);
\fill [black] (342 :7.4) circle (0.2cm);

\draw [black,very thick] (162:9) to (162:7);
\draw [black,very thick] (162:7) to (162:5);

\draw [black,very thick] (90:5) to (18:1.6);
\draw [black,very thick] (234:5) to (306:1.6);
\draw [black,very thick] (18:5) to (306:1.6);
\draw [black,very thick] (306:5) to (18:1.6);
\draw [black,very thick] (306:9) to (342:7.4);
\draw [black,very thick] (18:9) to (342:7.4);

\fill [black] (306 :1.6) circle (0.2cm);
\fill [black] (18 :1.6) circle (0.2cm);

\draw (-9.1,3) node[scale=0.5, myblue]   {$(34)$};
\fill [myblue] (162 :9) circle (0.2cm);
\fill [myblue] (162:11.1) circle (0.2cm);
\draw (-11,3.8) node[myblue]   {$L$};
\draw[myblue, very thick, dotted] (-9.6,3) circle (1.05 cm);

\draw (-4.7,2) node[scale=0.5, myred]  {$(02)$};
\draw (4.7,2) node[scale=0.5, myred]   {$(13)$};
\draw (0,9.5) node[scale=0.5, myred]   {$(01)$};
\draw (9.2,3) node[scale=0.5, myred]   {$(24)$};
\draw (-0.6,4.9)node[scale=0.5, myred]   {$(23)$};
\draw (3.6,-4)node[scale=0.5, myred]   {$(14)$};
\draw (-3.6,-4)node[scale=0.5, myred]   {$(04)$};
\draw (-5.8,-7.6)node[scale=0.5, myred]   {$(12)$};
\draw (5.8,-7.6)node[scale=0.5, myred]   {$(03)$};

\draw (0,2.1) node[scale=0.5, myred]  {$(02)(13)$};
\draw (-2.5,0.65) node[scale=0.5, myred]  {$(23)(04)$};
\draw (-1.4,-1.8) node[scale=0.5, myred]  {$(02)(14)$};
\draw (-1,7) node[scale=0.5, myred]  {$ (01)(23)$};
\draw (6.3,2.6) node[scale=0.5, myred]  {$ (24)(13) $};
\draw (-3.1,-5.8) node[scale=0.5, myred]  {$ (04)(12)$};
\draw (3.1,-5.8) node[scale=0.5, myred]  {$ (14)(03) $};

\draw (0,-7.9) node[scale=0.5, myred]  {$ (12)(03) $};
\draw (5.4,6) node[scale=0.5, myred]  {$ (01)(24) $};

\fill [myred] (234 :9) circle (0.2cm);
\fill [myred] (306 :9) circle (0.2cm);
\fill [myred] (18 :9) circle (0.2cm);
\fill [myred] (90 :9) circle (0.2cm);

\draw [myred ,line width=1mm] (90:9) to (54:7.4);
\draw [myred,line width=1mm] (234:9) to (270:7.4);
\draw [myred,line width=1mm] (306:9) to (270:7.4);

\draw [myred,line width=1mm] (18:9) to (54:7.4);

\fill [myred] (54 :7.4) circle (0.2cm);
\fill [myred] (270 :7.4) circle (0.2cm);

\fill [myred] (90 :7) circle (0.2cm);
\fill [myred] (234 :7) circle (0.2cm);
\fill [myred] (306 :7) circle (0.2cm);
\fill [myred] (18 :7) circle (0.2cm);

\draw [myred,line width=1mm] (90:9) to (90:7);
\draw [myred,line width=1mm] (234:9) to (234:7);
\draw [myred,line width=1mm] (306:9) to (306:7);
\draw [myred,line width=1mm] (18:9) to (18:7);
\draw [myred,line width=1mm] (90:7) to (90:5);
\draw [myred,line width=1mm] (234:7) to (234:5);
\draw [myred, line width=1mm] (306:7) to (306:5);
\draw [myred,line width=1mm] (18:7) to (18:5);

\fill [myred] (90 :5) circle (0.2cm);
\fill [myred] (162 :5) circle (0.2cm);
\fill [myred] (234 :5) circle (0.2cm);
\fill [myred] (306 :5) circle (0.2cm);
\fill [myred] (18 :5) circle (0.2cm);

\draw [myred,line width=1mm] (90:5) to (162:1.6);
\draw [myred,line width=1mm] (234:5) to (162:1.6);
\draw [myred,line width=1mm] (18:5) to (90:1.6);
\draw [myred,line width=1mm] (162:5) to (90:1.6);
\draw [myred,line width=1mm] (162:5) to (234:1.6);
\draw [myred,line width=1mm] (306:5) to (234:1.6);

\fill [myred] (90 :1.6) circle (0.2cm);
\fill [myred] (162 :1.6) circle (0.2cm);
\fill [myred] (234 :1.6) circle (0.2cm);
\draw (3.5,8) node[myred]  {$Red$};
\end{tikzpicture}\tag{Figure 5}
\end{equation}
The red thick subgraph corresponds to a divisor $D$ of type $\tilde{A}_{17}$. According to Corollary \ref{ellipticK3}, $D$ is a singular fiber of some elliptic fibration $p:X_4\ra \PP^1$. Note that $p$ is of type \textbf{(2)} with respect to Proposition \ref{ellipticX_4}. Indeed, ramification curve $L_{34}$ of $\pi$ does not intersect with $D$ hence it could not be a section of $p$. On the other hand any other ramification curve $L_{ij}$ of $\pi$ is contained in $D$. Thus sections of $D$ are not ramification curves of $\pi$ and this shows that $p$ is of type \textbf{(2)}. Since $p$ is of type \textbf{(2)} in Proposition \ref{ellipticX_4}, we derive that there exists reducible fiber $D'$ of $p$ distinct from $D$. Using notation as in discussion preceding Proposition \ref{ellipticX_4}, we derive that $p(D)$ and $p(D')$ are fixed points of $\tau$. Since $p$ is equivariant, we derive that $p(L_{34})$ is a fixed point of $\tau$. Thus $L_{34}\subseteq D'$. According to the fact that $D$ contains all ramification curves of $\pi$ except $L_{34}$ and due to Proposition \ref{classes}, fiber $D'$ must be of type $\tilde{A}_1$ with respect to Kodaira classification \cite[Chapter V, Section 7]{CCS}. Hence there exists a smooth, rational curve $L$ such that $L\cup L_{34}$ is of type $\tilde{A}_1$. We denote $L\cup L_{(34)}$ by blue dotted subgraph on the picture above. Clearly $L\in \bd{l_1}$.
\end{remark}
\noindent
We are ready to prove main results of this work.
\begin{fact}\label{central}
The nontrivial automorphism $\sigma:X_4\ra X_4$ of the covering $\pi$ is a central element of $\mathrm{Aut}(X_4)$.
\end{fact}
\begin{proof}
Recall that according to \cite[Chapter 3, Corollary 3.4]{Huy} the group of Hodge isometries of the transcendental lattice of a projective K3 surface is finite and cyclic. In particular, the image of $\sigma$ under $\mathrm{Aut}(X_4)\ra U_{X_4}$ is central. According to the fact that $\sigma^*=1_{\mathrm{Pic}(X_4)}=1_{S_{X_4}}$, we derive that $\sigma^*\in O_+(S_{X_4})$ is also central. According to Proposition \ref{cartesian} there is an isomorphism
$$\mathrm{Aut}(X_4)\cong U_{X_4}\times_{\mathrm{Aut}(D_{X_4})}O_+(S_{X_4})$$
It follows that $\sigma$ is a central element of $\mathrm{Aut}(X_4)$.
\end{proof}
\noindent
Let $f\in \mathrm{Aut}(X_4)$ be an automorphism. Then there exists a unique automorphism $\theta(f)\in \mathrm{Aut}(Y_4)$ such that $\pi \cdot \theta(f)=f\cdot \pi$. Indeed, $Y_4$ is a quotient of $X_4$ with respect to $\left<\sigma \right>=\{1_{X_4},\sigma\}\subseteq \mathrm{Aut}(X_4)$ and this is a central subgroup of $\mathrm{Aut}(X_4)$ by Fact \ref{central}. Hence every $f\in \mathrm{Aut}(X_4)$ is equivariant with respect to action of this subgroup. Finally every equivariant automorphism induces a unique automorphism of a quotient.

\begin{theorem}\label{exactaut}
There exists a short exact sequence of groups
\begin{center}
\begin{tikzpicture}
[description/.style={fill=white,inner sep=2pt}]
\matrix (m) [matrix of math nodes, row sep=3em, column sep=2em,text height=1.5ex, text depth=0.25ex] 
{  1 & \left<\sigma\right>&\mathrm{Aut}(X_4) & \mathrm{Aut}(Y_4)&1    \\} ;
\path[->,font=\scriptsize]  
(m-1-1) edge node[auto] {$ $} (m-1-2)
(m-1-2) edge node[auto] {$ $} (m-1-3)
(m-1-3) edge node[auto] {$ \theta$} (m-1-4)
(m-1-4) edge node[auto] {$ $} (m-1-5);
\end{tikzpicture}
\end{center}
and $\left<\sigma \right>$ is a central subgroup of $\mathrm{Aut}(X_4)$.
\end{theorem}
\begin{proof}
By the discussion above it remains to prove that $\theta$ is onto. Recall that $B_{\pi}$ denotes the branch divisor of $\pi$. Since $B_{\pi}=\sum_{0\leq i<j\leq 4}F_{ij}$ and $F_{ij}$ for $0\leq i<j\leq 4$ are the only $(-4)$-curves on $Y_4$ and any two among them do not intersect, we derive that $B_{\pi}$ is preserved by every automorphism of $Y_4$. Application of Proposition \ref{lifting} to the anticanonical cyclic covering $\pi$ shows that every automorphism $h$ of $Y_4$ admits a lift $\tilde{h}:X_4\ra X_4$ i.e there exists an automorphism $\tilde{h}:X_4\ra X_4$ such that $\pi\cdot \tilde{h}=h\cdot \pi$. Obviously $\theta(\tilde{h})=h$. This shows that $\theta$ is surjective.\\
Finally note that the kernel of $\theta$ consists of automorphisms of $X_4$ over $Y_4$. Hence it is equal to $\left<\sigma \right>$.
\end{proof}
\begin{remark}
{\cite[Theorem 2.4]{VIN}} implies that exact sequence in Theorem \ref{exactaut} does not admit a section.
\end{remark}
\noindent
Let $\Sigma_{\bd{b}}$ be the group of bijections of the set $\bd{b}$ of branch curves of $\pi$. The action of $\mathrm{Aut}(Y_4)$ on $(-4)$-curves of $Y_4$ gives rise to a homomorphism of groups $\mathrm{Aut}(Y_4)\ra \Sigma_{\bd{b}}$. We define a subgroup $G\subseteq \Sigma_{\bd{b}}$ as the image of an injective homomorphism $\Sigma_5\ra \Sigma_{\bd{b}}$ sending $\tau \in \Sigma_5$ to the bijection given by $F_{ij}\mapsto F_{\tau(i)\tau(j)}$.
\begin{theorem}\label{main3}
The image of the homomorphism $\mathrm{Aut}(Y_4)\ra \Sigma_{\bd{b}}$ is $G$. Moreover, the epimorphism $\mathrm{Aut}(Y_4)\ra G$ admits a section given by homomorphism $G\cong \mathrm{Aut}(S_5)\ra \mathrm{Aut}(Y_4)$ that lifts an automorphism of $S_5$ along birational contraction $m:Y_4\ra S_5$.
\end{theorem}
\begin{proof}
Recall that $Y_4$ is constructed as the blowing up $m:Y_4\ra S_5$ of intersection points of all $(-1)$-curves on $S_5$. Since every automorphism of $S_5$ permutes these intersection points, we derive that every automorphism of $S_5$ can be uniquely lifted to an automorphism of $Y_4$. This gives a homomorphism $s:G\cong \mathrm{Aut}(S_5)\ra \mathrm{Aut}(Y_4)$.\\
We define a graph $\Gamma$. Vertices of $\Gamma$ are primitive elements of the lattice $S_{X_4}$ with negative self-intersection. Two vertices $v_1$, $v_2$ of $\Gamma$ are adjacent if $\langle v_1, v_2 \rangle \neq 0$. The set of ramification curves $\bd{r}$ of $\pi$ is a subset of the set of vertices of $\Gamma$. We define a subgroup $\mathrm{Aut}(\bd{r},\Gamma)$ of $\mathrm{Aut}(\Gamma)$ consisting of these automorphisms of $\Gamma$ that preserve vertices in $\bd{r}$. Pick an automorphism $\phi\in \mathrm{Aut}(X_4)$. Then $\phi$ induces an orthogonal transformation of $S_{X_4}$ and hence it induces an automorphism $\phi^*$ of the graph $\Gamma$. Now according to the short exact sequence of Theorem \ref{exactaut} and the fact that $\pi$ sends curves in the class $\bd{r}$ to curves in the class $\bd{b}$ of all $(-4)$-curves on $Y_4$(Proposition \ref{classes}), we derive that $\phi^*$ preserves $\bd{r}$. Thus there exists a homomorphism of groups $\Phi:\mathrm{Aut}(X_4)\ra \mathrm{Aut}(\bd{r},\Gamma)$ given by $\phi \mapsto \phi^*$. We have the following diagram
\begin{center}
\begin{tikzpicture}
[description/.style={fill=white,inner sep=2pt}]
\matrix (m) [matrix of math nodes, row sep=3em, column sep=2em,text height=1.5ex, text depth=0.25ex] 
{  \mathrm{Aut}(X_4) & \mathrm{Aut}(\bd{r},\Gamma) &\mathrm{Aut}(\Gamma)     \\
                     & \Sigma_{\bd{r}}\cong \Sigma_{\bd{b}}        &         \\          } ;
\path[->,font=\scriptsize]  
(m-1-1) edge node[auto] {$ \Phi $} (m-1-2)
(m-1-2) edge node[auto] {$ \Xi $} (m-2-2);
\path[right hook->,font=\scriptsize]  
(m-1-2) edge node[auto] {$ $} (m-1-3);
\end{tikzpicture}
\end{center}
Here the second vertical arrow is the canonical monomorphism of groups and $\Xi$ is restriction to $\bd{r}$. Since $\sigma$ induces identity on $S_{X_4}$, we deduce that $\Phi(\langle \sigma \rangle)=\{1_{\Gamma}\}$. Thus we have a factorization $\Xi\cdot \Phi=\Psi\cdot \theta$, where $\Psi:\mathrm{Aut}(Y_4)\ra \Sigma_{\bd{r}}\cong \Sigma_{\bd{b}}$ is a homomorphism given by restricting automorphisms of $Y_4$ to the set of $(-4)$-curves and $\theta$ is defined in Theorem \ref{exactaut}. In {\cite[Section 2.2]{VIN}} it is stated that $\mathrm{Aut}(\Gamma)\cong \Sigma_5$. Therefore, $\Phi\left(\mathrm{Aut}(X_4)\right)\subseteq \mathrm{Aut}(\Gamma)\cong \Sigma_5$ and hence $\big| \Psi\left(\mathrm{Aut}(Y_4)\right)\big|\leq 5!=\big|G\big|$. Moreover, by construction of $s$ we have $\left(\Psi \cdot s\right)(G)=G$. Thus $\Psi$ has $G$ as its image and induces an epimorphism $\mathrm{Aut}(Y_4)\ra G$ that has $s$ as its section.
\end{proof}
\noindent
Now we come back to the construction of $X_4$ from section \ref{3}. In that section we constructed $X_4$ starting from four points $P=\{p_1,p_2,p_3,p_4\}$ on $\PP^2$ such that no three of them are on the same line. Hence there exists a system of homogeneous coordinates $[x_0,x_1,x_2]$ on $\PP^2$ such that $p_1=[-1,1,1]$, $p_2=[1,-1,1]$, $p_3=[1,1,-1]$, $p_4=[1,1,1]$. For given $\{i,j\}\in \{1,2,3,4\}$ denote by $N_{ij}$ the line through $p_k$ and $p_l$ for $\{k,l\}=\{1,2,3,4\}\setminus \{i,j\}$. Observe that $[1,1,-1]$, $[1,-1,1]$, $[-1,1,1]$ and $[1,1,1]$ are points of multiplicity three of this configuration of lines and there are three additional points of multiplicity two. One can easily calculate that these points are $q_1=[0,0,1]$, $q_2=[0,1,0]$, $q_3=[1,0,0]$. The whole configuration is showed in the following picture. 
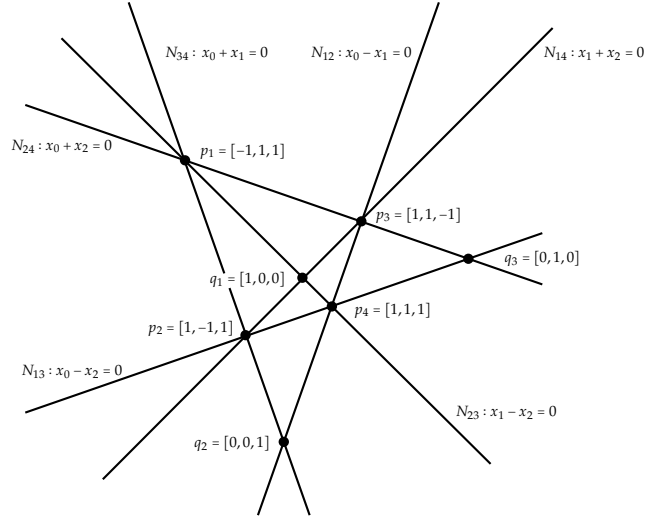
\begin{center}
\begin{tikzpicture}[scale=0.68, every node/.style={scale=0.8}]
\node[scale=0.7] at (-0.8,2.1) {$p_1=[-1,1,1]$};
\node[scale=0.7] at (-1.8,-1.35) {$p_2=[1,-1,1]$};
\node[scale=0.7] at (2.6,0.85) {$p_3=[1,1,-1]$};
\node[scale=0.7] at (2.1,-1) {$p_4=[1,1,1]$};
\node[scale=0.7] at (-0.7,-0.4) {$q_1=[1,0,0]$};
\node[scale=0.7] at (-1,-3.6) {$q_2=[0,0,1]$};
\node[scale=0.7] at (5,0) {$q_3=[0,1,0]$};
\node[scale=0.7] at (6,4) {$N_{14}:x_1+x_2=0$};
\node[scale=0.7] at (1.5,4) {$N_{12}:x_0-x_1=0$};
\node[scale=0.7] at (-1.3,4) {$N_{34}:\,x_0+x_1=0$};
\node[scale=0.7] at (4.3,-3) {$N_{23}:x_1-x_2=0$};
\node[scale=0.7] at (-4.3,2.2) {$N_{24}:x_0+x_2=0$};
\node[scale=0.7] at (-4.1,-2.2) {$N_{13}:x_0-x_2=0$};
\fill [black] (1.5,0.73) circle (0.1cm);
\fill [black] (0.93,-0.93) circle (0.1cm);
\fill [black] (-0.74,-1.49) circle (0.1cm);
\fill [black] (-1.91,1.92) circle (0.1cm);
\fill [black] (0.36,-0.37) circle (0.1cm);
\fill [black] (0,-3.57) circle (0.1cm);
\fill [black] (3.57,0) circle (0.1cm);
\draw [black, thick] (-5,3) to (5,-0.5);
\draw [black,  thick] (-5,-3) to (5,0.5);
\draw [black, thick] (-3,5) to (-1.15,-0.285);
\draw [black, thick] (-1.05,-0.571) to (0.5,-5);
\draw [black,  thick] (3,5) to (-0.5,-5);
\draw [black,  thick] (-4.3,4.3) to (4,-4);
\draw [black,  thick] (-3.5,-4.3) to (5.2,4.5);
\end{tikzpicture}
\end{center}
In the picture every point and line is described in terms of the homogeneous coordinates $[x_0,x_1,x_2]$.  For given $\{i,j\}\subseteq \{1,2,3\}$ denote by $K_{ij}$ the line through points $q_i$ and $q_j$. Now consider the standard quadratic transformation of $\PP^2$ given by formula $Q([x_0,x_1,x_2])=[x_1x_2,x_0x_2,x_0x_1]$. One verifies that
\begin{enumerate}[label=\textbf{(\roman*)}, leftmargin=*]
\item $Q$ is undefined precisely at points $q_1$, $q_2$, $q_3$ and blowing up of points $\{q_1,q_2,q_3\}$ resolves the indeterminacy of $Q$.
\item $Q(K_{ij}\setminus \{q_1,q_2,q_3\})=q_k$ for $\{i,j\}\subseteq \{1,2,3\}$ and $\{k\}=\{1,2,3\}\setminus \{i,j\}$.
\item $Q(N_{ij}\setminus \{q_1,q_2,q_3\})\subseteq N_{ij}$ for $\{i,j\}\subseteq \{1,2,3,4\}$ and $Q(p_i)=p_i$ for $i\in \{1,2,3,4\}$. 
\end{enumerate}

\begin{proposition}\label{quadratic}
Quadratic tranformation $Q$ of $\PP^2$ gives rise to an automorphism $f_Q$ of order two of the surface $Y_4$. This automorphism has the following properties.
\begin{enumerate}[label=\emph{\textbf{(\alph*)}}, leftmargin=*]
\item $f^*_Q:\mathrm{Pic}(Y_4)\ra \mathrm{Pic}(Y_4)$ acts as identity on some sublattice of $\mathrm{Pic}(Y_4)$ of rank $19$.
\item $f_Q$ stabilizes all curves $\{F_{ij}\}_{0\leq i<j\leq 4}$ and stabilizes all curves $\{F_{(ij)(kl)}\}$ for $\{i,j\}$, $\{k,l\}\subseteq \{0,1,2,3,4\}$, $\{i,j\}\cap \{k,l\}=\emptyset$ except $F_{(12)(34)}$, $F_{(13)(24)}$, $F_{(14)(23)}$. This is denoted in the \ref{fig6}, where filled vertices correspond to curves of the configuration that are stabilized by $f_{Q}$.
\begin{equation}\label{fig6}
\begin{tikzpicture}[scale=0.37, every node/.style={scale=0.8}]
\draw (-4.7,2) node[scale=0.5]  {$(02)$};
\draw (4.7,2) node[scale=0.5]   {$(13)$};
\draw (0,9.5) node[scale=0.5]   {$(01)$};
\draw (-9.2,3) node[scale=0.5]   {$(34)$};
\draw (9.2,3) node[scale=0.5]   {$(24)$};
\draw (-0.6,4.9)node[scale=0.5]   {$(23)$};
\draw (3.6,-4)node[scale=0.5]   {$(14)$};
\draw (-3.6,-4)node[scale=0.5]   {$(04)$};
\draw (-5.8,-7.6)node[scale=0.5]   {$(12)$};
\draw (5.8,-7.6)node[scale=0.5]   {$(03)$};

\draw (0,2.1) node[scale=0.5]  {$(02)(13)$};
\draw (-2.5,0.65) node[scale=0.5]  {$(23)(04)$};
\draw (2.5,0.65) node[scale=0.5]  {$(23)(14)$};
\draw (1.4,-1.8) node[scale=0.5]  {$(13)(04)$};
\draw (-1.4,-1.8) node[scale=0.5]  {$(02)(14)$};

\draw (-1,7) node[scale=0.5]  {$ (01)(23)$};
\draw (-6.3,2.6) node[scale=0.5]  {$ (34)(02)$};
\draw (6.3,2.6) node[scale=0.5]  {$ (24)(13) $};
\draw (-3.1,-5.8) node[scale=0.5]  {$ (04)(12)$};
\draw (3.1,-5.8) node[scale=0.5]  {$ (14)(03) $};

\draw (-5.4,6) node[scale=0.5]  {$ (01)(34) $};
\draw (5.4,6) node[scale=0.5]  {$ (01)(24) $};
\draw (-8,-2.4) node[scale=0.5]  {$ (34)(12) $};
\draw (8.05,-2.4) node[scale=0.5]  {$ (24)(03) $};
\draw (0,-7.9) node[scale=0.5]  {$ (12)(03) $};

\fill [black] (90 :9) circle (0.2cm);
\fill [black] (162 :9) circle (0.2cm);
\fill [black] (234 :9) circle (0.2cm);
\fill [black] (306 :9) circle (0.2cm);
\fill [black] (18 :9) circle (0.2cm);

\draw [black ,very thick] (90:9) to (54:7.4);
\draw [black, very thick] (90:9) to (126:7.4);
\draw [black,very thick] (162:9) to (126:7.4);
\draw [black,densely dotted] (162:9) to (198:7.4);
\draw [black,densely dotted] (234:9) to (198:7.4);
\draw [black,very thick] (234:9) to (270:7.4);
\draw [black,very thick] (306:9) to (270:7.4);
\draw [black,very thick] (306:9) to (342:7.4);
\draw [black,very thick] (18:9) to (342:7.4);
\draw [black,very thick] (18:9) to (54:7.4);

\fill [black] (54 :7.4) circle (0.2cm);
\fill [black] (126 :7.4) circle (0.2cm);
\fill [black] (198 :7.4) circle (0.2cm);
\fill [black] (270 :7.4) circle (0.2cm);
\fill [black] (342 :7.4) circle (0.2cm);

\fill [black] (90 :7) circle (0.2cm);
\fill [black] (162 :7) circle (0.2cm);
\fill [black] (234 :7) circle (0.2cm);
\fill [black] (306 :7) circle (0.2cm);
\fill [black] (18 :7) circle (0.2cm);

\draw [black,very thick] (90:9) to (90:7);
\draw [black,very thick] (162:9) to (162:7);
\draw [black,very thick] (234:9) to (234:7);
\draw [black,very thick] (306:9) to (306:7);
\draw [black,densely dotted] (18:9) to (18:7);

\draw [black,very thick] (90:7) to (90:5);
\draw [black,very thick] (162:7) to (162:5);
\draw [black,very thick] (234:7) to (234:5);
\draw [black, very thick] (306:7) to (306:5);
\draw [black,densely dotted] (18:7) to (18:5);

\fill [black] (90 :5) circle (0.2cm);
\fill [black] (162 :5) circle (0.2cm);
\fill [black] (234 :5) circle (0.2cm);
\fill [black] (306 :5) circle (0.2cm);
\fill [black] (18 :5) circle (0.2cm);

\draw [black,very thick] (90:5) to (162:1.6);
\draw [black,densely dotted] (90:5) to (18:1.6);
\draw [black,very thick] (234:5) to (162:1.6);
\draw [black,very thick] (234:5) to (306:1.6);
\draw [black,very thick] (18:5) to (306:1.6);
\draw [black,very thick] (18:5) to (90:1.6);
\draw [black,very thick] (162:5) to (90:1.6);
\draw [black,very thick] (162:5) to (234:1.6);
\draw [black,very thick] (306:5) to (234:1.6);
\draw [black,densely dotted] (306:5) to (18:1.6);

\fill [black] (90 :1.6) circle (0.2cm);
\fill [black] (162 :1.6) circle (0.2cm);
\fill [black] (234 :1.6) circle (0.2cm);
\fill [black] (306 :1.6) circle (0.2cm);
\fill [black] (18 :1.6) circle (0.2cm);

\fill [white] (18 :7) circle (0.15cm);
\fill [white] (18 :1.6) circle (0.15cm);
\fill [white] (198 :7.4) circle (0.15cm);
\end{tikzpicture}\tag{Figure 6}
\end{equation}
\end{enumerate}
\end{proposition}
\begin{proof}
Recall that $S_5=\mathrm{Bl}_P(\PP^2)$ for $P=\{p_1,p_2,p_3,p_4\}$. Consider the birational morphism $S_5\ra \PP^2$. As in the beginning of Section \ref{3}, denote by $E_{ij}$ the strict transform of $N_{ij}$ for $\{i,j\}\in \{1,2,3,4\}$ and let $E_{0i}$ be the exceptional curve mapping to the point $p_i$ for $i\in \{1,2,3,4\}$. Denote also  by $\ol{q}_1$, $\ol{q}_2$, $\ol{q}_3$ preimages under $S_5\ra \PP^2$ of points $q_1$, $q_2$, $q_3$. Since
$$\{q_1\}=N_{14}\cap N_{23},\,\{q_2\}=N_{12}\cap N_{34},\,\{q_3\}=N_{13}\cap N_{24}$$
we derive that
$$\{\ol{q}_1\}=E_{14}\cap E_{23},\,\{\ol{q}_2\}=E_{12}\cap E_{34},\,\{\ol{q}_3\}=E_{13}\cap E_{24}$$
Finally let $\ol{K}_{ij}$ for every $\{i,j\}\in \{1,2,3\}$ be the strict transform of $K_{ij}$ to $S_5$. Clearly $Q$ lifts to a birational automorphism $\ol{Q}$ of $S_5$ and $\ol{Q}$ has the following properties \textbf{(1)}-\textbf{(3)} corresponding to properties \textbf{(i)}-\textbf{(iii)} of $Q$. 
\begin{enumerate}[label=\textbf{(\arabic*)}, leftmargin=*]
\item $\ol{Q}$ is undefined precisely at points $\ol{q}_1$, $\ol{q}_2$, $\ol{q}_3$ and blow up of $\ol{q}_1$, $\ol{q}_2$, $\ol{q}_3$ resolves the indeterminacy of $\ol{Q}$.
\item $\ol{Q}(\ol{K}_{ij}\setminus  \{\ol{q}_1,\ol{q}_2,\ol{q_3}\})=\ol{q}_k$ for $\{i,j\}\subseteq \{1,2,3\}$ and $\{k\}=\{1,2,3\}\setminus \{i,j\}$.
\item $\ol{Q}(E_{ij}\setminus \{\ol{q}_1,\ol{q}_2,\ol{q_3}\})\subseteq E_{ij}$ for every $\{i,j\}\in \{0,1,2,3,4\}$.
\end{enumerate}
Since $Y_4$ is obtained from $S_5$ via blowing up intersection points of configuration $\{E_{ij}\}$ for $0\leq i<j\leq 4$ and due to 
$$\{\ol{q}_1\}=E_{14}\cap E_{23},\,\{\ol{q}_2\}=E_{12}\cap E_{34},\,\{\ol{q}_3\}=E_{13}\cap E_{24}$$
we derive using \textbf{(1)} that there exists a morphism $r_Q:Y_4\ra S_5$ such that the following diagram 
\begin{center}
\begin{tikzpicture}
[description/.style={fill=white,inner sep=2pt}]
\matrix (m) [matrix of math nodes, row sep=3em, column sep=2em,text height=1.5ex, text depth=0.25ex] 
{  Y_4  &      \\
S_5 & S_5\\} ;
\path[densely dotted,->,font=\scriptsize]  
(m-2-1) edge node[below] {$\ol{Q} $} (m-2-2);
\path[->,font=\scriptsize]  
(m-1-1) edge node[left] {$m $} (m-2-1)
(m-1-1) edge node[auto] {$ r_Q$} (m-2-2);
\end{tikzpicture}
\end{center}
is commutative. Obviously $r_{Q}$ contracts all curves $F_{(ij)(kl)}$ for $\{i,j\}$, $\{k,l\}\subseteq \{0,1,2,3,4\}$, $\{i,j\}\cap \{k,l\}=\emptyset$ except $F_{(12)(34)}$, $F_{(13)(24)}$, $F_{(14)(23)}$ and according to \textbf{(2)} preimage under $r_{Q}$ of each of points $\ol{q}_1$, $\ol{q}_2$, $\ol{q}_3$ is one-dimensional. Thus preimage under $r_{Q}$ of each intersection point of configuration $\{E_{ij}\}$ for $0\leq i<j\leq 4$ is one-dimensional. Therefore, using \cite[Chapter 4, Proposition 5.3]{HAR}, we derive that $r_{Q}$ factors through blow up of $S_5$ at all intersection points of configuration $\{E_{ij}\}_{0\leq i<j\leq 4}$. Since this blow up is $Y_4$, we derive that there exists a morphism $f_{Q}$ such that the following diagram
\begin{center}
\begin{tikzpicture}
[description/.style={fill=white,inner sep=2pt}]
\matrix (m) [matrix of math nodes, row sep=3em, column sep=2em,text height=1.5ex, text depth=0.25ex] 
{  Y_4  & Y_4      \\
S_5 & S_5\\} ;
\path[densely dotted,->,font=\scriptsize]  
(m-2-1) edge node[below] {$\ol{Q} $} (m-2-2);
\path[->,font=\scriptsize]  
(m-1-1) edge node[left] {$m $} (m-2-1)
(m-1-2) edge node[auto] {$ m$} (m-2-2)
(m-1-1) edge node[auto] {$ f_Q$} (m-1-2);
\end{tikzpicture}
\end{center}
is commutative. Next according to the fact that as a rational map $Q\cdot Q$ is equal to $1_{\PP^2}$, we derive that as a rational map $f_Q\cdot f_Q$ is equal to $1_{Y_4}$. Since $f_Q$ is a morphism, we deduce that $f_{Q}\cdot f_{Q}=1_{Y_4}$ i.e. $f_Q$ is an algebraic automorphism and $f_{Q}^2=1_{Y_4}$. If $f_{Q}$ is trivial, then $Q$ will be trivial and this is not the case. Hence $f_{Q}$ is a nontrivial involution of $Y_4$.\\
Now we show that $f_Q^*:\mathrm{Pic}(Y_4)\ra \mathrm{Pic}(Y_4)$ acts as identity on some sublattice of $\mathrm{Pic}(Y_4)$ of rank $19$.  According to \textbf{(1)} and \textbf{(3)}, we derive that $f_{Q}$ stabilizes all curves $\{F_{ij}\}_{0\leq i<j\leq 4}$ and stabilizes all curves $\{F_{(ij)(kl)}\}$ for $\{i,j\}$, $\{k,l\}\subseteq \{0,1,2,3,4\}$, $\{i,j\}\cap \{k,l\}=\emptyset$ except $F_{(12)(34)}$, $F_{(13)(24)}$, $F_{(14)(23)}$. Thus $f^*_Q$ acts as identity on a sublattice of $\mathrm{Pic}(Y_4)$ generated by
$$\{F_{ij}\}_{0\leq i<j\leq 4},\,\{F_{(ij)(kl)}\mid 0\leq i<j\leq 4,\,0\leq k<l\leq 4,\{i,j\}\cap \{k,l\}=\emptyset \}\setminus \{F_{(12)(34)},F_{(13)(24)},F_{(14)(23)}\}$$
Next note that similarly to the situation in Proposition \ref{transc}, there exists a fibration $S_5\ra \PP^1$ with general fiber being smooth, rational curve having three singular fibers 
$$E_{14}+E_{23},\,E_{12}+E_{34},\,E_{13}+E_{24}$$
Precomposing this fibration with $m:Y_4\ra S_5$, we obtain a fibration $Y_4\ra \PP^1$ having precisely three singular fibers  $F_{(14)}+F_{(23)}+2F_{(14)(23)},\,F_{(12)}+F_{(34)}+2F_{(12)(34)},\,F_{(13)}+F_{(24)}+2F_{(13)(24)}$. Thus we have linear equivalences
$$2F_{(14)(23)}-2F_{(12)(34)}\sim \left(F_{(14)}+F_{(23)}\right)-\left(F_{(12)}+F_{(34)}\right)$$
$$2F_{(14)(23)}-2F_{(13)(24)}\sim \left(F_{(14)}+F_{(23)}\right)-\left(F_{(13)}+F_{(24)}\right)$$
Since $f_Q$ stabilizes curves $\{F_{ij}\}_{0\leq i<j\leq 4}$, we deduce that $f^*_Q$ acts as identity on 
$$F_{(14)(23)}-F_{(12)(34)},\, F_{(14)(23)}-F_{(13)(24)}$$ 
Construction of $Y_4$ via birational morphisms $m:Y_4\ra S_5$ and $S_5\ra \PP^2$ shows that the sublattice of $\mathrm{Pic}(Y_4)$ generated by 
$$\{F_{ij}\}_{0\leq i<j\leq 4},\,\{F_{(ij)(kl)}\mid 0\leq i<j\leq 4,\,0\leq k<l\leq 4,\{i,j\}\cap \{k,l\}=\emptyset \}\setminus \{F_{(12)(34)},F_{(13)(24)},F_{(14)(23)}\}$$
and by $F_{(14)(23)}-F_{(12)(34)},\, F_{(14)(23)}-F_{(13)(24)}$ is of rank $19$. 
\end{proof}
\noindent
Denote by $\tilde{f}_Q:X_4\ra X_4$ any lift of $f_Q:Y_4\ra Y_4$ along homomorphism described in Theorem \ref{exactaut}. 

\begin{corollary}\label{reflectionX_4}
Consider the model of a hyperbolic space described in Corollary \ref{hyperbolic} specified to $X_4$. The linear map
$$\tilde{f}_Q^*:\mathrm{H}^{1,1}_{\RR}(X_4)\ra \mathrm{H}^{1,1}_{\RR}(X_4)$$
induced by a lift $\tilde{f}_Q:X_4\ra X_4$ of $f_Q:Y_4\ra Y_4$ gives rise to a reflection of this hyperbolic space.
\end{corollary}
\begin{proof}
According to Proposition \ref{quadratic}, action of an automorphism $\tilde{f}_Q$ on $\mathrm{H}^{1,1}_{\RR}(X_4)$ leaves  invariant vectors of some subspace of codimension one. Moreover, $\tilde{f}_Q^*:\mathrm{H}^{1,1}_{\RR}(X_4)\ra \mathrm{H}^{1,1}_{\RR}(X_4)$ is a linear isometry. Thus $\tilde{f}^*_Q$ is an isometric linear endomorphism of $\mathrm{H}^{1,1}_{\RR}(X_4)$ leaving  invariant vectors of a hyperplane and preserving the ample cone. This implies that there exists a vector $e\in \mathrm{H}^{1,1}_{\RR}(X_4)$ such that $(e,e)<0$ and 
$$\tilde{f}^*_Q(x)=x-\frac{2(e,x)}{(e,e)}e$$
By Proposition \ref{reflection}, $\tilde{f}^*_Q$ induces a reflection of the hyperbolic space described in Corollary \ref{hyperbolic}.
\end{proof}

\begin{corollary}\label{main4}
Reflection induced by $\tilde{f}_Q^*$ on the hyperbolic space associated with $\mathrm{H}^{1,1}_{\RR}(X_4)$ is conjugate as an element of $O(S_{X_4})$ to reflections contained in $\mathscr{S}_1$ according to Vinberg's notation \cite[Section 2.2]{VIN}. 
\end{corollary}
\begin{proof}
In \cite[Section 2.2]{VIN} author claims that there are three conjugacy classes of linear reflections inside $O(S_{X_4})$ that induce reflections of the hyperbolic space associated with $\mathrm{H}^{1,1}(X_4)$. He chooses three sets of these reflections $\mathscr{S}_2'$, $\mathscr{S}''_2$ and $\mathscr{S}_1$ each contained in precisely one conjugacy class inside $O(S_{X_4})$. Then he shows that only reflections conjugate to those in the set $\mathscr{S}_1$ induce automorphisms of $X_4$.
\end{proof}

\section*{Acknowledgements}

I would like to express my gratitude to my advisor Jarosław Wiśniewski for his constant support and patient guidance. I wish to thank Joachim Jelisiejew who made valuable editorial and mathematical suggestions that vastly improve the presentation of the material.

\bibliography{mybib}{}
\bibliographystyle{amsalpha}

\end{document}